\documentclass[12pt]{article}  % default square logo 
\usepackage{amssymb}
 
\usepackage{lscape}								% Landscape tables 
\usepackage[dvipsnames]{xcolor}					% provides colors in document
\usepackage{color}								% allows coloring
\usepackage[toc,title]{appendix}				% provide appendix section

\usepackage{comment}							% provides writing comments	
\usepackage[numbers]{natbib}					% enhanced citing options
\usepackage{pdfpages}                           % allows integrating pdf files
\usepackage{xspace}								% define commands that appear not to eat spaces
\usepackage{epstopdf} 							% converts EPS to PDF using Ghostscript
\usepackage{nomencl}							% enable Nomenclature
\usepackage{etoolbox}							% enable groupgings in Nomencaltura
\usepackage{longtable}							% tables spanning pages
\usepackage{authblk} 

\usepackage[ margin=1in]{geometry}

\renewcommand\nomgroup[1]{%
	\item[\bfseries
	\ifstrequal{#1}{M}{Matrices}{%
		\ifstrequal{#1}{V}{Vectors}{%
			\ifstrequal{#1}{Sets}{Sets}{}}}%
	]}

\usepackage{graphicx}							% graphic environment
\usepackage{multirow}							% multiple rows in lists
\usepackage{paralist}							% enhanced lists
\usepackage{verbatim}							% discards latex commands and shows signs as the are (e.g. easily write source codes)
\usepackage{longtable}							% allows tables over several pages
\usepackage{booktabs}							% controls lines in tables
\usepackage{subcaption}							% combines two figures as float object

\makeatletter
\def\BState{\State\hskip-\ALG@thistlm}
\makeatother

\usepackage{tikz}								% create figures
\usetikzlibrary{graphs,graphs.standard,quotes}	% graph environments
\usetikzlibrary{backgrounds}					% include backgrounds				
\usetikzlibrary{fit} 							% do not know
\usepackage[T1]{fontenc}
\usepackage{lmodern}
\usepackage[utf8]{inputenc}						% umlaute
\usepackage{soul}								% strike out words
\usepackage[normalem]{ulem}						% underline or strike out words
\usepackage{mathrsfs}							% scripture-like letters

\usepackage[linesnumbered,noend]{algorithm2e}   % [linesnumbered,noend]
\usepackage{setspace}
\DontPrintSemicolon
\SetKw{KwAnd}{and}
\SetKw{KwOr}{or}
\SetKw{KwTrue}{true}
\SetKw{KwFalse}{false}

\usepackage{forloop}							% provides the \forloop for writing makros
\usepackage{ifthen}								% provides the if-then loop

\colorlet{darkred}{red!80!black}
\colorlet{darkgreen}{green!60!black}
\colorlet{darkblue}{blue!80!black}
\colorlet{darkorange}{orange!70!black}
\definecolor{Green}{cmyk}{1, 0.2, 0.4, 0.1}
\definecolor{Orange}{cmyk}{0, 0.61, 0.87, 0}
\definecolor{Purple}{rgb}{0.75, 0.0, 1.0}
\definecolor{purple}{rgb}{0.5,0,0.5}
\definecolor{dgreen}{rgb}{0.1,0.9,0.5}
\definecolor{gabysgreen}{cmyk}{0.80, 0.1, 0.90, 0}

\newcommand{\blue}[1]{{\color{black}#1}}

\usepackage{amssymb}							% enhanced math
\usepackage[fleqn]{amsmath}						% enhanced math
\usepackage{amsthm}								% math theorems
\usepackage{bm}									% bold math
\usepackage{sansmath}
\usepackage{mathtools}							% allows at least definition sign
\usepackage{nicefrac}							% allows 1/p
\usepackage{siunitx}							% si-units for writing scientific numbers
\usepackage{xfrac}								% for nicer fractions 

\usepackage{hyperref}							% referencing with hyperlinks
\hypersetup{
	colorlinks,
	linkcolor={blue!50!black},
	citecolor={green!50!black},
	urlcolor={blue!80!black}
}
\usepackage[capitalise,noabbrev]{cleveref}		% clever referencing (e.g., automatically adds name of reference)

\newcommand{\irg}[2]{[#1\!:\!#2]}
\newcommand{\inirgs}{i \in \irg{1}{S}}
\newcommand{\R}{\mathbb{R}}						%real numbers
\newcommand{\N}{\mathbb{N}}						%natural numbers
\newcommand{\F}{\FF}					%feasible set
\newcommand{\CPP}{\mathcal{CPP}}
\newcommand{\CPI}{\mathcal{CPI}}				%completely positive partial intersection matrices
\newcommand{\CPS}{\mathcal{CPS}}				%completely positive partial sum
\newcommand{\CBC}{\mathcal{CBC}}				%completely positive block clique
\newcommand{\DNN}{\mathcal{DNN}}

\newcommand{\CMP}{\mathcal{CMP}}				%CompletelyPositive Completable Components
\newcommand{\x}{\vc{x}}							%bold vector x
\newcommand{\lrbr}[1]{\left\lbrace #1 \right\rbrace} %Left and Right braces

\newcommand\gl{\lambda}

\def\x{\mathbf x}
\def\aa{\mathbf a}
\def\y{\mathbf y}
\def\z{\mathbf z}

\def\v{\mathbf v}
\def\w{\mathbf w}

\def\f{\mathbf f}
\def\oo{\mathbf o}
\def\e{\mathbf e}

\def\cc{\mathbf c}

\def\g{\mathbf g}

\def\r{{\mathbf r}}

\def\AA{{\mathcal A}}

\newcommand\CC{{\mathcal C}}

\newcommand\EE{{\mathcal E}}
\newcommand\FF{{\mathcal F}}

\newcommand\II{{\mathcal I}}

\newcommand\KK{{\mathcal K}}

\newcommand\NN{{\mathcal N}}

\newcommand\OO{{\mathcal O}}

\def\SS{{\mathcal S}}

\newcommand\VV{{\mathcal V}}

\newcommand\Ab{{\mathsf{A}}}
\newcommand\Bb{{\mathsf B}}
\newcommand\Cb{{\mathsf C}}

\newcommand\Fb{{\mathsf F}}
\newcommand\Gb{{\mathsf G}}
\newcommand\Hb{{\mathsf H}}
\newcommand\Ib{{\mathsf I}}

\newcommand\Mb{{\mathsf M}}

\newcommand\Ob{{\mathsf O}}

\newcommand\Qb{{\mathsf Q}}

\newcommand\Wb{{\mathsf W}}
\newcommand\Xb{{\mathsf X}}
\newcommand\Yb{{\mathsf Y}}
\newcommand\Zb{{\mathsf Z}}

\newcommand{\scp}[2]{\langle #1, #2 \rangle}						% scalar product || other option: {\langle #1, #2 \rangle}
\newcommand{\tr}[1]{\mathrm{Tr}(#1)}								% trace
\newcommand{\T}{^\mathsf{T}}	
\DeclareMathOperator{\conv}{conv}									% convex hull
\DeclareMathOperator{\diag}{diag}									% diagonalize matrix to vector
\def\beq#1{$$}														% begin equation
\def\eps{\varepsilon}

\def\bea#1{\begin{array}{#1}}
	\def\ea{\end{array}}
\def\ignore#1{}

\theoremstyle{plain}
\newtheorem{thm}{Theorem} % [section]
\newtheorem{exmp}{Example} %[section]

\newtheorem{lem}[thm]{Lemma}

\newtheorem{rem}{Remark}
\theoremstyle{definition}
\theoremstyle{remark}

\usepackage{fancyhdr}

\begin{document}

	%TITLE TR
	\title{Sparse Conic Reformulation of Structured QCQPs based on Copositive Optimization with Applications in Stochastic Optimization %\thanks{Supported by the
%Vienna Science and Technology Fund (project ICT15-014), and the
%Austrian Science Fund (project DK W 1260-N35).}
}
	
	\author{Markus Gabl} %\thanks{markus.gabl@univie.ac.at}}
	% %
	\affil{IOR, Karlsruhe Institute of Technology, Germany.\\ \texttt{markus.gabl@kit.edu}\\
%\texttt{+43-1-4277-[38652|838672]}
   }
	\renewcommand\Authands{ and }
	\maketitle

	\begin{abstract}
		Recently, Bomze et.\ al.\ introduced a sparse conic relaxation of the scenario problem of a two stage stochastic version of the standard quadratic optimization problem. When compared numerically to Burer's classical reformulation, the authors showed that there seems to be almost no difference in terms of solution quality, whereas the solution time can differ by orders of magnitudes. While the authors did find a very limited special case, for which Burer's reformulation and their relaxation are equivalent, no satisfying explanation for the high quality of their bound was given. This article aims at shedding more light on this phenomenon and give a more thorough theoretical account of its inner workings. We argue that the quality of the outer approximation cannot be explained by traditional results on sparse conic relaxations based on positive semidenifnite or completely positive matrix completion, which require certain sparsity patterns characterized by chordal and block clique graphs respectively, and put certain restrictions on the type of conic constraint they seek to sparsify. In an effort to develop an alternative approach, we will provide a new type of convex reformulation of a large class of stochastic quadratically constrained quadratic optimization problems that is similar to Burer's reformulation, but lifts the variables into a comparatively lower dimensional space. The reformulation rests on a generalization of the set-completely positive matrix cone. This cone can then be approximated via inner and outer approximations in order to obtain upper and lower bounds, which potentially close the optimality gap, and hence can give a certificate of exactness for these sparse reformulations outside of traditional, known sufficient conditions. Finally, we provide some numerical experiments, where we asses the quality of the inner and outer approximations,  thereby showing that the approximations may indeed close the optimality gap in interesting cases.  
		

	\end{abstract}	
	\pagebreak	
	 
	\section{Introduction}\label{chap:ReducingtheSizeofConicReformulations} 
	Recently, in \cite{bomze_two-stage_2022}, the authors considered the scenario problem of a two-stage stochastic version of the standard quadratic optimization problem given by 
	\begin{align}\label{eqn:St3QP} \tag{2St3QP}
		\min_{\x\in\R^{n_1},\y_i\in\R^{n_2}} \lrbr{\x\T\Ab\x+ \sum_{i=1}^{S}p_i \left(\x\T\Bb_i\y_i + \y_i\T\Cb_i\y_i\right) \colon (\x,\y_i)\in \Delta, \ \inirgs},	
	\end{align}
	where $\Delta\subset\R^{n_1+n_2}$ is the unit simplex, and $p_i,\ \inirgs$ are probabilities of certain scenarios occurring. This optimization problem can be exactly reformulated into a copositive optimization problem based on Burer's reformulation presented in \cite{burer_copositive_2009}. The reformulation forces a lifting of the space of variables into a space of dimension $O((n_1+Sn_2)^2)$, which makes this reformulation entirely impractical for the purposes of stochastic optimization since the number of scenarios $S$ is typically very high and the copostive optimization problem has to be approximated with semidefinite optimization problems, which are known to scale poorly. In an effort to circumvent this issue, the authors introduced a copositive relaxation that merely requires $O(S(n_1+n_2)^2)$ variables and showed empirically that the approximating SDPs are practical even if the number of scenarios is high. Somewhat surprisingly, they observed that the quality of the solutions they found did not substantially differ from the bound obtained by employing the traditional copositive reformulation. In fact they state that the difference was small enough to possibly be an artifact of numerical inaccuracies of the sdp-solver. Aside from an exactness result for a niche case of (\ref{eqn:St3QP}) no theoretical explanation for this phenomenon was provided. The present article is chiefly motivated by the question: why does the cheap relaxation perform so well? While we were not able to fully answer this question, we are still able to provide valuable theoretical insights that amount to a novel, practical approach to sparse conic reformulations. In short we introduce a generalization of the set-completely positive matrix cones that yield conic relaxations that are sparse to begin with, and which can, much like the traditional set-completely positive matrix cones, be approximated in order to generate lower and upper bounds, that may certify optimality in case the gap between them is zero.   
	
	To set up our exposition, we will now introduce a more general quadratic optimization problem and discuss some important context, specifically copositive optimization and sparse conic reformulations based on matrix completion. To begin with, the class optimization problems in question is given by: 
	\begin{align}\label{eqn:DecomposableQuadraticProblem}
		\begin{split}
			\min_{\x,\y_i} \x\T\Ab\x+ \aa\T\x &+ \sum_{i=1}^{S} \left[\x\T\Bb_i\y_i + \y_i\T\Cb_i\y_i + \cc_i\T\y_i\right]\\
			\mathrm{s.t.:}\  \Fb_i\x+\Gb_i\y_i &= \r_i, \quad  \hspace{0.06cm} i\in\irg{1}{S},\\
			Q_{j}(\x,\y_1,\dots,\y_S) &= 0, \quad \hspace{0.1cm} j\in\irg{1}{K},\\
			\x&\in \KK_0,\\
			\y_i&\in\KK_i ,\quad i\in\irg{1}{S},
		\end{split}
	\end{align}  
	where $\KK_0\subseteq\R^{n_1},\ \KK_i\subseteq\R^{n_2},\ i \in\irg{1}{S}$, are closed, convex cones, $\Ab\in\SS^{n_1}$ (i.e.\ symmetric matrices of order $n_1$), $\ \aa\in\R^{n_1}$, $\Bb_i\in \R^{n_1\times n_2},\ \Cb_i\in\SS^{n_2},\cc_i\in\R^{n_2},\ i \in \irg{1}{S}$ and $\Fb_i\in\R^{m_i\times n_1},\ \Gb_i\in\R^{m_i\times n_2},\ \r_i\in\R^{m_i} ,\ i \in\irg{1}{S}$. Further, $Q_{j}(\cdot)\colon \R^{n_1+Sn_2} \rightarrow \R,\ j\in\irg{1}{K}$ are quadratic functions that do not involve bilinear terms between $\y_i$ and $\y_j$ for $i\neq j$. The special structure in place here is that $\y_i$ does not interact with $\y_j$ in a bilinear fashion in neither the constraints nor the objective, and the statement stays true even if the linear constraints are squared.   
	
	This setup encompasses not only (\ref{eqn:St3QP}), but general two-stage stochastic conic QCQPs over finitely supported distributions, which are important since they are used to approximate two-stage stochastic conic QCQPs with infinite support. In the context of two-stage stochastic optimization, $S$ would be the number of scenarios and $\y_i$ would be variables specific to scenario $i$. Hence, the special structure in (\ref{eqn:DecomposableQuadraticProblem}) is native to all two-stage stochastic QCQPs regardless of the structure of the nominal QCQP. %\red{We will hint at another possible application of this model later in the text}

	Under some well known regularity conditions on the functions $Q_{i}(.)$ (see \cite{burer_copositive_2012,kim_geometrical_2020,eichfelder_set-semidefinite_2013}) our structured QCQP can be reformulated into a conic optimization problem with $\OO\left((n_1+Sn_2)^2\right)$ variables. This reformulation takes the form: 
	\begin{align}\label{eqn:DecomposableQCQPBurer}
		\begin{split}
			\min_{\Xb,\Yb_i,\Zb_i,\x,\y_i} \tr{\Ab_i\Xb }+ \aa\T\x&+\sum_{i=1}^{S} \left[\tr{\Bb_i\Zb_i}+ \tr{\Cb_i\Yb_{i,i}} +\cc_i\T\y\right]\\
			\mathrm{s.t.:}\  \Fb_i\x+\Gb_i\y_i &= \r_i, \quad \hspace{0.7cm} i\in\irg{1}{S},\\
			\diag\left(
			\begin{pmatrix}
				\Fb_i & \Gb_i
			\end{pmatrix}
			\begin{pmatrix}
				\Xb & \Zb_i\T\\\Zb_i & \Yb_i
			\end{pmatrix}
			\begin{pmatrix}
				\Fb_i\T\\ \Gb_i\T
			\end{pmatrix}\right) &= \r_i\circ\r_i, \quad i\in\irg{1}{S},\\
			\hat{Q}_{j}(\x,\Xb,\y_1,\Zb_1,\Yb_1,\dots,\y_S,\Zb_S,\Yb_S) &= 0, \quad \hspace{0.62cm} \ j\in\irg{1}{K},\\
			\begin{pmatrix}
				1 & \x\T &\y_1\T & \dots & \y_S\T \\
				\x& \Xb & \Zb_1\T & \dots & \Zb_S\T \\
				\y_1&\Zb_1 & \Yb_{1,1} & \dots & \Yb_{1,S}\\
				\vdots&\vdots&\vdots&\ddots & \vdots \\
				\y_S & \Zb_S & \Yb_{S,1}&\dots & \Yb_{S,S}
			\end{pmatrix}   &\in \CPP(\R_+\times_{i=0}^S\KK_i).
		\end{split}
	\end{align}
	 for appropriate linear functions $\hat{Q}_{ij}$ with $\hat{Q}_{ij}(\x,\x\x\T,\y_1,\y_1\x\T,\y_1\y_1\T,\dots,\y_S,\y_S\x\T,\y_S\y_S\T) = Q_{ij}(\x,\y_1,\dots,\y_S)$, with $\circ$ denoting the elementwise multiplication of vectors, and the set-completely positive matrix cone is define as 
	\blue{
	\begin{align*} 
		\CPP_n(\KK) &\coloneqq \left\lbrace \Xb \colon \Xb = \sum_{i=1}^{k}\x_i\x_i\T,  \x_i \in \KK,\ i\in \irg{1}{k}, \ k\in\N   \right\rbrace\\ &= \mathrm{clconv}\left\lbrace \x\x\T \colon \x \in \KK \right\rbrace = \{\Xb\Xb\T \colon \colon  \Xb\in \R^{n\times k}\, , \,  \x_i \in \KK, \, i \in\irg{1}{k},\ k\in\N  \},
	\end{align*}}
	for a closed, convex cone $\KK\subseteq\R^n$. For example, $\CPP(\R^n)$ is the positive semidefinite matrix cone, denoted by $\SS^n_+$, and  $\CPP(\R^n_+)$ is the classical completely positive matrix cone, extensively discussed in \cite{berman_completely_2003}. In the literature, optimization over the set-completely positive cone and its dual, the set-copositive matrix cone is colloquially referred to as copositive optimization. In general, set-completely positive matrix cones are intractable and have to be approximated. For example, it is well known that 
	\begin{align*}
		\CPP(\R^n_+) \subseteq \DNN^n \coloneqq \SS^n_+\cap\NN^n,
	\end{align*}
	where $\NN^n$ is the cone of nonnegative $n\times n$ matrices and $\DNN^n$ is called the doubly nonnegative matrix cone. \blue{In the above set containment, equality holds whenever $n\leq 4$.}
	
	While many tractable approximations do exist, be it based on positive semidefinite, second order cone or linear programming constraints, they all have in common that their complexity increases exponentially with the approximations quality. Even simple approximations, such as $\DNN^n$, typically involve semidefinite constraints of the same order as the set-completely positive constraint. As a result the above reformulation is often impractical. Especially in the context of stochastic optimization, where the number of scenarios $S$ is typically very high, the size of the psd-constraints, which is of the order $\OO\left((n_1+Sn_2)^2\right)$, becomes prohibitive. 
	
	Following the basic idea of the authors in \cite{bomze_two-stage_2022}, we can obtain a lower dimensional relaxation by replacing the conic constraint by $S$ smaller conic constraints given by
	\begin{align}\label{eqn:SmallerConicConstraints}
		\begin{pmatrix}
			1 & \x\T &\y_i\T \\
			\x& \Xb & \Zb_i\T  \\
			\y_i & \Zb_i & \Yb_{i}
		\end{pmatrix}  \in \CPP\left(\R_+\times\KK_i\right), \ \inirgs,
	\end{align}
	 \blue{where $\Yb_{i,i}$ was relabeled $\Yb_i$ and $\Yb_{i,j}, \ i \neq j$ were dropped from the model}, so that the number of variables is now $\OO\left(S(n_1+n_2)^2\right)$, therefore linear in $S$.
	 The cost we have to pay is that the resulting optimization problem is not necessarily equivalent to (\ref{eqn:DecomposableQCQPBurer}) and hence to (\ref{eqn:DecomposableQuadraticProblem}), as the conic constraints are clearly a relaxation of the conic constraint of the exact reformulation. It is, however, this relaxation which performed so inexplicably well when it was applied in \cite{bomze_two-stage_2022}. 
	
	In searching for an explanation of this performance, one may turn to the literature on sparse reformulations of conic optimization problems. Results in this field are typically based on theorems on matrix completion. The central question in that area is, when a given matrix with non-specified entries, so called partial matrices, can be completed to a matrix in $\CPP(\KK)$. This is useful in the context of solving a conic optimization problem: if the problem data is sparse so that some of the entries of the matrix variable only appear in the conic constraint, one can check if the removal of that entries leaves a partial matrix for which one can give sufficient conditions so that it is completable to a matrix that fulfills the original conic constraint. If such conditions are available, they replace the conic constraint and the spurious entries of the matrix can be dropped entirely. The result is what we will call a sparse reformulation of the original conic problem. These reformulations can reduce the number of variables substantially, which eases the computational burden so that otherwise unmanageable problems become viable. 
	
	The classical text on such an approach is \cite{vandenberghe_chordal_2015}, where the conic constraint to be massaged is an sdp-constraint. Their approach utilizes the fact that a partial matrix, where the non-specified entries exhibit the so called chordal sparsity pattern can always be completed to a psd-matrix provided all fully specified, principle  submatrices are positive semidefinite. The framework was applied in various contexts such as robust optimization \cite{padmanabhan_exploiting_2021} or optimal power flow \cite{fan_achieving_2018,sliwak_improving_2019}. \blue{In \cite{zheng_chordal_2020} the authors harnessed sparsifications of SDPs computationally by exploiting properties of the ADMM algorithm}. Another approach was recently put forward by \cite{kim_doubly_2020}, who applied classical $\CPP(\R^n_+)$-completion results derived in \cite{drew_completely_1998} in the context of copositive optimization. Their approach necessitates the presence of so called block-clique sparsity patterns in the problem data, owing to the fact that partial matrices with block-clique specification pattern can be completed to matrices in $\CPP(\R^n_+)$ whenever the fully specified, principle submatrices are completely positive. \blue{On a related note, \cite{bettiol_mining_2022} explored sparse convex reformulations of binary quadratic optimization problems based on matrix completion where the set of matrices of interest was not some version of $\CPP$ but the so called Boolean Quadratic Polytope, i.e.\ the convex hull of matrices $\x\x\T$ where $\x\in\lrbr{0,1}^n$. Interestingly, completion results for this subset of $\CPP(\R^n_+)$ seem to be stronger than for general matrices in this cone, which bolsters our hope of finding stronger completion results for other such subsets as well.}
	
	Unfortunately, none of these results are able to explain the phenomenon we seek to investigate and we will spend a full section on discussing their shortcomings and what we can still learn from them about our object of interest. We will argue that, unless $\KK = \R^n$, the required sparsity patterns are, outside of some limited special cases, not the ones present in (\ref{eqn:DecomposableQCQPBurer}), where the sparsity pattern takes the form of an arrow-head. Also, in cases where $\KK$ is neither the positive orthant nor the full space, completion results are, to the best of our knowledge, entirely absent from literature.

	\subsubsection*{Contribution}
	In an effort to remedy these shortcomings, we propose a new approach to sparse conic reformulations. Rather than treating completability of a matrix as an abstract concept we identify a cone that is isomorphic to the cone of completeable partial matrices with arrowhead sparsity pattern, denoted $\CMP$, as a generalization of the set-completely positive matrix cone. We show that the geometry of this cone can be used in order to derive a lower dimensional alternative to the exact reformulation (\ref{eqn:DecomposableQCQPBurer}). Much the same way one uses inner and outer approximations in order to solve copositive optimization problems, we derive inner and outer approximations of $\CMP$ in order to obtain upper and lower bounds to this new conic optimization problem. Numerical experiments show that in practice these approximations exhibit interesting beneficial properties. 
	
	\subsubsection*{Outline}
	The rest of the article is organized as follows: In \cref{sec: Classical approaches to sparse conic optimization and why they fail} we will give a short discussion on existing approaches to sparse conic optimization and discuss the limitations that ultimately make these techniques unfit to tackle sparse reformulations of (\ref{eqn:DecomposableQCQPBurer}). Hence, we develop an alternative approach in \cref{sec:An alternative approach to sparse reformulations}, based on the aforementioned convex cone $\CMP$. This new type of convex reformulation motivates a strategy to sparse optimization that is analogous to classical copositive optimization techniques, where difficult conic constraints are approximated via inner and outer approximations. In \cref{sec:Inner and  outer approximation of CMP based on set-completely positive matrix cones} we present many such approximations and discuss their limitations. Finally, we asses the efficacy of our approach in extensive numerical experiments.

%	\begin{itemize}
%		\item We introduce an new type of sparse convex relaxation that involves a generalization of the set-completely positive matrix cone that is called the cone completely positive, completable, connected components ($\CMP$), which lifts the variables into space with dimesion of $\OO\left(S(n_1+n_2)^2\right)$. 
%		\item We give a geometric argument to proof that this relaxation is in fact exact. The argument is a new way of applying the theory laid out in \cite{kim_geometrical_2020}. 
%		\item Since we cannot optimize over the $\CMP$ directly we provide an inner and outer approximations both in terms of $\CPP(\KK)$. Thus, all the knowledge we have on the latter cone can be directly applied to our new reformulation. The outer approximation is in fact equivalent to the sparse relaxation proposed in \cite{bomze_two-stage_2022}, the inner approximation is entirely new.   
%		\item We show that traditional arguments for establishing a sparse reformulation cannot be applied to establish equivalence between (\ref{eqn:DecomposableQCQPBurer}) and the our proposed outer approximation reformulation outside of niche special cases.
%		\item We provide numerical experiments that \red{werden wir sehen}
%	\end{itemize}

	\subsection*{Notation}
	Throughout the paper matrices are denoted by sans-serif capital letters (e.g.\ $\Ob$ will denote the zero matrix, where the size will be clear from the context), vectors by boldface lower case letters (e.g.\ $\oo$ will denote the zero vector, $\e_i$ will denote a vector of zeros with a one at the $i$-th coordinate) and scalars (real numbers) by simple lower case letters. Sets will be denoted using calligraphic letters, e.g., cones will often be denoted by $\KK$. 
	We use $\SS^n$ to indicate the set of symmetric matrices and  $\SS^n_+$/$\SS^n_-$ for the sets of positive-/negative-semidefinite symmetric matrices, respectively. Moreover, we use $\NN_n$ to denote the set of entrywise nonnegative, symmetric matrices. We also use the shorthand notation $\irg{l}{k}\coloneqq \lrbr{l,l+1,\dots,k-1,k}\subseteq \N$. For a given set $\AA$ we denote its convex hull by $\mathrm{conv}(\AA)$. For a convex set $\CC$, the set of generators of its extreme rays and points is given by $\mathrm{ext}(\CC)$. \blue{Also, for a cone $\KK\subseteq\R^n$ we denote as 
	\begin{align*}
		\KK^*\coloneqq \lrbr{\x\in\R^n\colon \y\T\x \geq 0, \ \forall \y \in \KK },
	\end{align*}
	its dual cone.} We also make use of the Frobenius product of two appropriately sized matrices $\Ab$ and $\Bb$ defined as $\Ab\bullet\Bb \coloneqq \mathrm{trace}(\Ab\T\Bb)$, which can be interpreted as the sum of the inner products of the columns of $\Ab$ and $\Bb$.

 	\section{Classical approaches to sparse conic optimization and why they fail}\label{sec: Classical approaches to sparse conic optimization and why they fail}
 	
 	As stated in the introduction, there are already many approaches for utilizing sparsity patterns in conic optimization problems. At the core of these results lie matrix completion theorems, which we will discuss shortly. But in order to state them we must introduce some essential terms first. 
 	
 	A graph $G =(\VV,\EE)$ is given by its set of vertices $\VV =  \lrbr{v_1,\dots,v_n}$ and its set of edges
 	$\EE\subseteq\lrbr{\lrbr{v,u}\colon v,u\in\VV}$,
 	both of which are finite. A subgraph $T = (\VV_T,\EE_T)$ of a graph $G$ is a graph such that $\VV_T\subseteq\VV$ and $\EE_T\subseteq \EE$. Vertex $v_j$ is \textit{adjacent} to $v_j$ and vice versa if $\lrbr{v_i,v_j} \in \EE$. If $e= \{v_i,v_j\}\in\EE$ then $v_i$ and $v_j$ are \textit{incident} on $e$. A graph where all vertices are adjacent to one another is called a \textit{complete} graph. A \textit{path} that connects vertex $v_i$ with $v_j$ is given by a sequence of edges so that \blue{$\lrbr{\{v_i,v_{k_1}\},\dots,\{v_{k_{p-1}},v_j\}}\subseteq\EE$, $v_{k_i}$} are distinct and $p>1$ is the \textit{length} of that path. A graph is \textit{connected} if any two vertices have a connecting path. A graph that is not connected is \textit{disconnected}. A \textit{cycle} is path that connects a vertex $v$ to itself. A \textit{chord} of a cycle with length greater than 3 is an edge that connects two vertexes who are incident on two different edges of the cycle. A graph is \textit{chordal} if every cycle with length greater 3 has a chord. \blue{A collection of vertices in $\VV$ that induce a subgraph of $G$ that is complete is called a \textit{clique}}. A \textit{block} $B$ of a graph is a subgraph  that is connected, has no disconnected subgraph that is obtained by removing just one vertex and its adjacent edges (i.e.\ a \textit{cut vertex}) from $B$, and is not contained in any other subgraph with these two properties. A \textit{block-clique graph} is a  graph whose blocks are cliques. 
 	
 	A \textit{partial matrix} of order $n$ is a matrix whose entries in the $i$-th row and the $j$-th column are determined if and only if $(i,j)\in\II\subseteq \irg{1}{n}^2$ and are undetermined otherwise. A matrix is said to be \textit{partial positive semidefinite/ completely positive/ doubly nonnegative} if and only if every fully determined principal submatrix is  positive semidefinite/ completely positive/ doubly nonnegative. A partial matrix is \textit{positive semidefinite/ completely positive/ doubly nonnegative completable} if we can specify the undetermined entries so that the fully specified matrix is semidefinite/ completely positive/ doubly nonnegative.
 	
 	The \textit{specification graph} of partial matrix $\Ab$ of order $n$ is a graph $G(\Ab)$ with vertices $\VV = \lrbr{ v_i\colon i\in \irg{1}{n}}$ and edges $\EE$ such that $\lrbr{v_i,v_j}\in\EE$ if and only if the entry $a_{ij}$ is specified. A symmetric matrix with $G(\Ab) = G$  is called a \textit{symmetric matrix realization} of $G$. 
 	
 	The following three theorems give the key results on matrix completion as far as this text is concerned:
 	
 	\begin{thm}\label{thm:SDPCompletionChordalGraphs}
 		All partial positive semidefinite symmetric matrix realizations of a graph $G$ are positive semidefinite completable if and only if $G$ is chordal. 
 	\end{thm}
 	\begin{proof}
 		See \cite[Theorem 1.39]{berman_completely_2003}. 
 	\end{proof}
 	\begin{thm}\label{thm:CPPCompletionBlockCliqueGraphs}
 		Every partial completely positive matrix realization of a graph	$G$ is completely positive completable if and only if G is a block-clique graph
 	\end{thm}
 	\begin{proof}
 		See \cite[Theorem 2.33]{berman_completely_2003}.
 	\end{proof}
 	\begin{thm}\label{thm:DNNCompletionBlockCliqueGraphs}
 		Every partial doubly nonnegative matrix realization of a graph $G$ is doubly nonnegative completable if and only if G is a block-clique graph
 	\end{thm}
 	\begin{proof}
 		See \cite{drew_completely_1998}.  
 	\end{proof}

 	These theorems can be used in order to establish that a  constraint on a high-dimensional matrix, say $\Xb$, can be replaced by a number of constraints on certain principal submatrices of $\Xb$ without increasing the feasible set. This is achieved by showing that values for the submatrices of $\Xb$ that fulfill the latter constraints can be completed to a full evaluation of $\Xb$ that fulfills the original larger constraint. For the sake of illustration we present the following toy example.   
 	\begin{exmp}
 		Consider the optimization problem 
 		\begin{align}\label{eqn:ExmpProblem}
 			\min_{\Xb\in\SS^{n}_+} \lrbr{\Qb\bullet \Xb \colon \Bb\bullet \Xb = 1 },
 		\end{align}
 		where $(\Qb)_{ij} = (\Bb)_{ij} = 0$ if $|i-j|>1$, the remaining entries of $\Bb$ equal one and those of $\Qb$ are arbitrary. The entries of $\Xb$ that are outside the inner band of width 1 are not present in neither the equality constraint nor the objective. Consider the relaxation of (\ref{eqn:ExmpProblem}) where the psd-constraints is replaced by 
 		\begin{align}\label{eqn:ExmpConstraints}
 			\begin{pmatrix}
 				X_{ii} & X_{ij} \\ X_{ji} & X_{jj}
 			\end{pmatrix}\in \SS^2_+  \quad \forall (i,j) \colon |i-j| =  1,\ i<j,
 		\end{align}
 		and the entries of $\Xb$ outside of the inner band are dropped from the problem. Clearly, we obtain a relaxation of the original problem since the new condition is necessary for the $\Xb$ to be positive semidefinite. Also, the fact that we dropped entries of $\Xb$ can be thought of as a replacement of the matrix $\Xb$ by a partial matrix, say $\Xb_{*}$, whose entries outside the inner band are not specified. In this case the specification graph of $\Xb_{*}$ is easily checked to be chordal, as it doesn't contain any cycles at all. Hence, if all fully specified submatrices of  $\Xb_{*}$ are positive semidefinite, i.e. (\ref{eqn:ExmpConstraints}) holds, then it can be completed to positive semidefinite matrix by \cref{thm:SDPCompletionChordalGraphs}. The resulting matrix would be feasible for (\ref{eqn:ExmpProblem}) with the same objective function value, so that the relaxation turns out to be lossless.   
 	\end{exmp}

 	One may attempt to similarly derive a sparse reformulation of  (\ref{eqn:DecomposableQCQPCMPRefromulation}) by invoking the completion results we discussed above. This would necessitate to show that a partial matrix of the following form 
 	\begin{align*}
 		\hspace{-1cm}
 		\begin{pmatrix}
 			\Xb & \Zb_1\T & \Zb_2\T &\dots & \Zb_{S-1}\T & \Zb_S\T \\
 			\Zb_1 & \Yb_{1,1} & \mathbf{*} & \dots  & \mathbf{*}  & \mathbf{*}\\
 			\Zb_2 & \mathbf{*} & \Yb_{2,2}& \dots  & \mathbf{*}  & \mathbf{*} \\
 			\vdots&\vdots&\vdots&\ddots & \vdots & \vdots  \\
 			\Zb_{S-1}  & \mathbf{*}  & \mathbf{*} & \dots & \Yb_{S-1,S-1}  & \mathbf{*} \\
 			\Zb_S  & \mathbf{*}  & \mathbf{*}&\dots  & \mathbf{*} & \Yb_{S,S}
 		\end{pmatrix},\
 	\end{align*}
 	can be completed to a set-completely positive matrix whenever the submatrices 
 	\begin{align*}
 		\begin{pmatrix}
 			\Xb & \Zb_i\T  \\
 			\Zb_i & \Yb_{i,i}
 		\end{pmatrix}  \in \CPP\left(\KK_0\times\KK_i\right), \ \inirgs.
 	\end{align*}
 	Note, that this would coincide with the model in \cite{bomze_two-stage_2022}, which we discussed in the introduction, so that the matrix completion theory is a promising contender for the desired explanation for the effectiveness of the model. 
 	The strategy appears feasible at first, at least for the case where $\KK_i$ are nonnegative orthants given that in this case, completion results are readily available. Unfortunately it is futile, since the arrowhead structure is not block-clique outside of narrow special cases, as we will now show.

 	\begin{lem}\label{lem:ArrowHeadMatricesAreChordal}
 		Let $S>1$ and consider a partial matrix  where the specified entries exhibit an arrow-head structure, i.e. 
 		\begin{align*}
 			\hspace{-1cm}
 			\begin{pmatrix}
 				\Xb & \Zb_1\T & \Zb_2\T &\dots & \Zb_{S-1}\T & \Zb_S\T \\
 				\Zb_1 & \Yb_{1,1} & \mathbf{*} & \dots  & \mathbf{*}  & \mathbf{*}\\
 				\Zb_2 & \mathbf{*} & \Yb_{2,2}& \dots  & \mathbf{*}  & \mathbf{*} \\
 				\vdots&\vdots&\vdots&\ddots & \vdots & \vdots  \\
 				\Zb_{S-1}  & \mathbf{*}  & \mathbf{*} & \dots & \Yb_{S-1,S-1}  & \mathbf{*} \\
 				\Zb_S  & \mathbf{*}  & \mathbf{*}&\dots  & \mathbf{*} & \Yb_{S,S}
 			\end{pmatrix},\
 		\end{align*}
 		where $ \Xb\in\SS^{n_1},\ \Yb_{i,i}\in\SS^{n_2},\ \Zb_i\in\R^{n_2\times n_1}, \ i \in\irg{1}{S},$ and let $G_{spec}$ be its specification graph. Then $G_{spec}$ is chordal. If $n_1\in\lrbr{0,1}$ then $G_{spec}$ is also a block-clique graph, which is not the case otherwise. 
 	\end{lem}
 	\begin{proof}
 		We start out by showing that $G_{spec}$ is chordal in general. We group the nodes of the specification graph into $S+1$ groups where the first group $g_0 = \lrbr{1,\dots n_1}$ are the nodes that correspond to the first $n_1$ rows of the matrix and whose internal edges are specified by the north west entries $\Xb$. The second group $g_1 = \lrbr{n_1+1,\dots,n_1+n_2}$ corresponds to the rows $n_1+1$ to $n_1+n_2$ whose internal edges are specified by the blocks $\Yb_{2,2}$ and whose external edges, connecting to neighbors outside of $g_1$, are specified by $\Zb_1$. The construction of the remaining groups proceeds accordingly. We will now show that any cycle of length greater 3 must have a chord. Note that all the groups are cliques since the blocks $\Yb_{i,i}$ are fully specified. 
 		Thus, a cycle of length greater than 3 must have a chord if it is entirely contained in one of the groups.  We therefore only need to consider cycles that are not entirely contained in one group.	Also, any member of $g_0$ is a neighbor to any other node in the graph since the blocks $\Zb_i, \ i \in \irg{1}{S}$ are fully specified. Thus, if a vertex $v$ of $g_0$ is visited by a cycle, then the edge to any other node in the cycle that is not the predecessor of $v$ gives a chord.
 		A cycle that visits more than one group needs to visit $g_0$ since the other groups are not connected to one another and thus has a chord.
 		
 		If $n_1=1$ then $g_0$ is a singleton. A block cannot contain just vertices from multiple $g_i, \ i \in\irg{1}{S}$ since these groups are pairwise disconnected. A connection can only be established by adding $g_0$ but then the single node in $g_0$ is a cut vertex, i.e. the  subgraph can become disconnected by deleting a single node and its adjacent edges. Hence, a block of $G_{spec}$ must be a subgraph formed from the union of $g_0$ and one $g_i,\ i \in \irg{1}{S}$ and  the respective edges. A subgraph formed from all the nodes of such a union, say $T$,  cannot be contained in any other block since the construction of such a block would require to add nodes from a third group. Thus, $T$ is a block, but it is also a clique since the $q_i$ is a clique and the node in $g_0$ is adjacent to all the members of $g_i$. 
 		
 		If $n_1 = 0$ then $G_{spec}$ consists of $S$ subgraphs that are cliques and pairwise disconnected, hence they are blocks. 
 		
 		Otherwise, the entire graph is its only block since it cannot become disconnected by deleting a single node and its adjacent edges, but this block is not a clique since $g_i, \ i \in \irg{1}{S}$ have no inter-group edges.  
 	\end{proof}

	As a consequence of the lemma, the traditional route for sparse conic reformulations provides little insight: If $\KK_i$ are positive orthants the completion theorems are not applicable since (\ref{eqn:DecomposableQCQPBurer}) lacks the proper sparsity pattern. Also in that case, we cannot compare the $\DNN$ relaxations of (\ref{eqn:DecomposableQCQPBurer}) and its sparse relaxation based on (\ref{eqn:SmallerConicConstraints}), since the same sparsity pattern would be required. If $\KK_i$ are neither the positive orthant nor the full space, we do not even have any completion results to begin with. 
	
	Still, the present methodology allows for at least some insight into the benefits of working with (\ref{eqn:SmallerConicConstraints}), namely in the form of the following performance guarantee.  
	
 	\begin{thm}\label{thm:DecomposableQCQPSDPguarantee}
 		Let $\mathrm{val}(SDP)$ be the optimal value of problem (\ref{eqn:DecomposableQCQPBurer}) after $\CPP(\R_+\times_{i=0}^S\KK_i)$ is replaced by $\SS^{n_1+Sn_2+1}_+$ and let $\mathrm{val}(R)$ be that optimal value after replacing the full conic constraint with the conic constraints in (\ref{eqn:SmallerConicConstraints}). 
 		We have $\mathrm{val}(SDP)\leq\mathrm{val}(R)$ and the statement also holds if we replace the cones $ \CPP(\R_+\times \KK_0\times\KK_i),\ i\in\irg{1}{S}$ in (\ref{eqn:SmallerConicConstraints}) by any other subsets of $\SS^{n_1+n_2+1}_+$.
 	\end{thm}
 	\begin{proof}
 		Clearly, the two problems have the same objective function, so we only need to compare the feasible sets. Let $\left(\Xb,\Yb_1,\dots,\Yb_S,\Zb_i,\dots,\Zb_S,\x,\y_1,\dots,\y_S\right)$ be such that 
 		\begin{align}\label{eqn:FullySpecifiedBlocks}
 			\begin{pmatrix}
 				1 & \x\T &\y_i\T \\
 				\x& \Xb  &\Zb_i\T\\
 				\y_i& \Zb_i& \Yb_i
 			\end{pmatrix} &\in\blue{\CPP(\R_+\times \KK_0\times\KK_i) \subseteq \SS^{n_1+n_2+1}_+,} \  i\in\irg{1}{S},
 		\end{align}
 		and the linear constraints in (\ref{eqn:DecomposableQCQPBurer}) are fulfilled, i.e.\ we have a feasible solution for the optimization problem defining $\mathrm{val}(R)$. \blue{The setinclusion holds since, $\CPP(\KK_A)\subseteq\CPP(\KK_B)$ whenever $\KK_A\subseteq\KK_B$, and $\SS^n_+ =\CPP(\R^n)$}. All we need to show, is that, after setting $\Yb_{i,i} = \Yb_i,\ i \in \irg{1}{S}$, we can find $\Yb_{i,j}, \ i \neq j$ such that we can construct a positive semidefinite matrix. By \cref{thm:SDPCompletionChordalGraphs}  it suffices to show that the specification graph of the partial matrix where $\Yb_{i,j},\ i\neq j$ are not specified is a chordal graph and that all fully specified principal submatrices are positive semidefinite. So consider the partial matrix 
 		\begin{align*}
 			\begin{pmatrix}
 				1 & \x\T &\y_1\T & \y_2\T & \dots&\y_{S-1}\T & \y_S\T \\
 				\x& \Xb & \Zb_1\T & \Zb_2\T &\dots & \Zb_{S-1}\T & \Zb_S\T \\
 				\y_1&\Zb_1 & \Yb_{1,1} & \mathbf{*} & \dots  & \mathbf{*}  & \mathbf{*}\\
 				\y_2&\Zb_2 & \mathbf{*} & \Yb_{2,2}& \dots  & \mathbf{*}  & \mathbf{*} \\
 				\vdots&\vdots&\vdots&\vdots&\ddots & \vdots & \vdots  \\
 				\y_{S-1} & \Zb_{S-1}  & \mathbf{*}  & \mathbf{*} & \dots & \Yb_{S-1,S-1}  & \mathbf{*} \\
 				\y_S & \Zb_S  & \mathbf{*}  & \mathbf{*}&\dots  & \mathbf{*} & \Yb_{S,S}
 			\end{pmatrix}.
 		\end{align*}
 		Since in all but the first two columns (we use this word now referring to literal columns in the above representation) have unspecified blocks one can only obtain fully specified principal submatrices if one deletes all but one of the partially specified columns and all but the respective rows (again in the literal sense). The so obtained blocks are precisely the blocks in (\ref{eqn:FullySpecifiedBlocks}) and are thus positive semidefinite. The chordality of the specification graph follows from \cref{lem:ArrowHeadMatricesAreChordal}. This completes the proof. 
 	\end{proof}
 	
 	The theorem states that our sparse, hence low dimensional, relaxation is at least as strong as the fully dimensional SDP-relaxation and thus gives a theoretical performance guarantee. It also applies to relaxations of (\ref{eqn:SmallerConicConstraints}) such as the $\DNN$-relaxation since $\DNN^n\subseteq \SS^n_+$. 
 	\begin{rem}
 		We could have arrived at \cref{lem:ArrowHeadMatricesAreChordal} by using the results in \cite{fukuda_exploiting_2001} who describe a chordality-detection procedure for SDPs with chordal sparsity pattern. \blue{The chordality of the arrow head matrices was also observed in \cite{bettiol_mining_2022,zheng_chordal_2020}. However, it is more convenient for the reader if this technical detail is proved here directly}. It is nonetheless important to note that the above result is not the first of its kind, but can be obtained directly from known results in literature. Still, to the best of our knowledge, the context in which we use this technique is original. 
 	\end{rem}

 	\begin{rem}
 		At this point we would also like to highlight a specific shortcoming of the above completion theorems. An inattentive reading of their claims might give the false impression that, as an example, for a partial psd-matrix to be completable, it needs to have a chordal specification graph. This assessment is incorrect. A partial psd-matrix $\Mb$ may have a specification graph $G(\Mb)$ that is not chordal, while still being psd-completeable. All the theorem says is that not all partial psd-matrices with specification graph $G(\Mb)$ are psd-completable. But that does not exclude the possibility that some still can be completed. This is significant, since for a sparse relaxation to be exact it suffices that its optimal set contains just one appropriately completable matrix. To additionally require that all other feasible matrices, or more so, all matrices with the same sparsity pattern are completable is needlessly restrictive, which explains part of the inflexibility of the classical machinery. 
 	\end{rem}
	
	\section{An alternative approach to sparse reformulations}\label{sec:An alternative approach to sparse reformulations}
	
	We have seen that the classical approach to sparse reformulations is limited in several capacities. It is restrictive with respect to the cones $\KK_i$ and it is inflexible with respect to the sparsity structure, such that it is ultimately ill-equipped to tackle sparse reformulations of (\ref{eqn:DecomposableQCQPBurer}). We therefore propose and alternative strategy, where we provide a convex-conic reformulation of (\ref{eqn:DecomposableQuadraticProblem}) based on a generalization of $\CPP$ that rests on a lifting of the space of variables into a space that is of lower dimension than required for the classical $\CPP$-reformulation (\ref{eqn:DecomposableQCQPBurer}). Hence, the reformulation is already sparse, which comes at the price of having to optimize over a new, complicated cone. This, however, is just a new guise of an old problem in copositive optimization, and we will meet in a, thus, familiar fashion: by providing inner and outer approximations, that provide upper and lower bounds on the problem whose gap is hopefully small or even zero. In order to achieve this we will first introduce some necessary concepts, that will allow us to state and proof our main reformulation result. After that, we close this section with a detailed description of our approach.

	\subsection{The space of connected components $\SS_n^{S,k}$ and the cone of completable, completely positive, connected components $\CMP$}\label{sec:The space of connected}
	We define 
	\begin{align*}
		\SS_n^{S,k} &\coloneqq \lrbr{ 
			\left[
			\begin{pmatrix}
				\Xb & \Zb_1\T \\ \Zb_1 & \Yb_1
			\end{pmatrix},\dots,
			\begin{pmatrix}
				\Xb & \Zb_S\T \\ \Zb_S & \Yb_S
			\end{pmatrix} \right]
			\colon 
			\begin{pmatrix}
				\Xb & \Zb_i\T \\ \Zb_i & \Yb_i
			\end{pmatrix}\in \SS^n,\ i \in \irg{1}{S}, \ \Xb\in\SS^k		
		},
	\end{align*}
	i.e. the set of vectors of $S$ symmetric matrices of order $n$ connected by a component of order $k$, which we call the space of connected components. In order to distinguish elements of $\SS_n^{S,k}$ from normal matrices we use san-serif letters braced by rectangular braces, for example $\left[\Ab\right]$. Note, that $\SS_n^{S,k}$ is isomorphic to the space of arrowhead matrices by the isomorphism 
	\begin{align*}
		\Gamma  \colon \SS^{S,k}_n \rightarrow \SS^{k+Sn}, \
		\left[\Ab\right]\mapsto \Gamma\left(\left[\Ab\right]\right) \coloneqq
		\begin{pmatrix}
			\Xb & \Zb_1\T &\dots & \Zb_S \\ 
			\Zb_1 & \Yb_1  & \dots & \Ob \\
			\vdots& \vdots & \ddots & \vdots \\
			\Zb_S & \Ob & \dots  & \Yb_S
		\end{pmatrix},  
	\end{align*}
	where for the inverse we have 
	\begin{align*}
		\Gamma^{-1}\begin{pmatrix}
			\Xb & \Zb_1\T &\dots & \Zb_S \\ 
			\Zb_1 & \Yb_1  & \dots & \Ob \\
			\vdots& \vdots & \ddots & \vdots \\
			\Zb_S & \Ob & \dots  & \Yb_S
		\end{pmatrix} = \left[
		\begin{pmatrix}
			\Xb & \Zb_1\T \\ \Zb_1 & \Yb_1
		\end{pmatrix},\dots,
		\begin{pmatrix}
			\Xb & \Zb_S\T \\ \Zb_S & \Yb_S
		\end{pmatrix} \right]\in \SS^{S,k}_n.
	\end{align*}
	Thus, $\SS_n^{S,k}$ is a vector space with a natural inner product $\left[\Ab\right]\odot\left[\Bb\right] \coloneqq \Gamma\left(\left[\Ab\right]\right)\bullet\Gamma\left(\left[\Ab\right]\right)$, sum $\left[\Ab\right]\oplus\left[\Bb\right] \coloneqq \Gamma\left(\left[\Ab\right]\right)+\Gamma\left(\left[\Ab\right]\right)$ and scalar multiplication $\gl\left[\Ab\right] \coloneqq \Gamma^{-1}\left(\gl \Gamma\left(\left[\Ab\right]\right)  \right)$. For notational convenience we will expand the meaning of the inverse $\Gamma^{-1}$ so that it is applicable to non-arrowhead matrices as well, where the nonzero off-diagonal blocks are treated as though they were blocks of zeros as in the definition above. Also, we define a second, analogous isomorphism $\Gamma_*(\cdot)$ that maps into the space of partial matrices where the blocks of zeros in the definition of $\Gamma(\cdot)$ are not specified. We also will use the shorthand notation
	\begin{align*}
		\left[
		\begin{pmatrix}
			\Xb & \Zb_1\T \\ \Zb_1 & \Yb_1
		\end{pmatrix},\dots,
		\begin{pmatrix}
			\Xb & \Zb_S\T \\ \Zb_S & \Yb_S
		\end{pmatrix} \right] = 
		\left[
		\begin{pmatrix}
			\Xb & \Zb_i\T \\ \Zb_i & \Yb_i
		\end{pmatrix}\right]_{i\in\irg{1}{S}}.
	\end{align*}

	The central object we are interested in is the following subset of $\SS_n^{S,k}$: 
	\begin{align*}
		\CMP \left(\KK_0,\dots,\KK_S\right) &\coloneqq \mathrm{conv}\lrbr{ 
			\left[
			\begin{pmatrix}
				\x \\ \y_i
			\end{pmatrix}
			\begin{pmatrix}
				\x \\ \y_i
			\end{pmatrix}\T
			\right]_{i\in\irg{1}{S}}
			\colon 
			\begin{pmatrix}
				\x \\ \y_i
			\end{pmatrix}\in\KK_0\times \KK_i,\ i \in \irg{1}{S}	
		},
	\end{align*}
	where $\KK_0\subseteq \R^{k},\ \KK_i\subseteq \R^{n-k},\ i \in \irg{1}{S}$ are convex cones, which we refer to as \textit{ground cones}. 
	%Since in most cases discussed in this text $\KK_i, \ \inirgs$ are identical we will often denote $\KK_0\otimes\KK \coloneqq (\KK_0,\KK\dots,\KK)$ as w vector of $S+1$ cones, to be used as an argument for $\CMP$. 
	We often use $\CMP$ without its arguments as a colloquial term, in case the respective ground cones are not important to, or clear from, the context at hand. The same is true for all abbreviations of its inner and outer approximations that will be discussed later in the text.  
	
	We call $\CMP$ the cone of \textit{completable, completely positive, connected components} and we will justify that name in a latter section. Further, we define  $\mathrm{gen} \CMP$ to be the set of its generators, i.e.\ the set we obtain by omitting the $\conv$-operator in the definition of $\CMP$. \blue{We also like to note, that for the case $S = 1$ the cone $\CMP$ reduces to $\CPP$, with its ground cone given by $\KK_0\times\KK_1$. Further, if $k=n$  the ground cone reduces to $\KK_0$,  so that $\CMP$ can be seen as a generalization of $\CPP$. However, we will see later in the text, that in fact a certain type of correspondence holds between the two objects (see \cref{rem:RemarkonCMP}).}

%	\begin{rem}
%		In order to keep the discussion in this text concise we provide a detailed discussion and analysis in a separate paper, namely \red{ZITAT}	
%	\end{rem}

	\subsection{Main result: a new type of convex reformulation, with reduced dimension}\label{sec:AgeometricalApproachTheCMP}
	
	The derivation of our main result relies heavily on the very general framework from \cite{kim_geometrical_2020}, for achieving convex reformulations for a large array of problems.  In the following paragraphs we will give a small and simplified account of their results in order to make the derivation of our main result as transparent as possible. The two theorems we discuss shortly are specializations of theorems in \cite{kim_geometrical_2020}, which we prove here for the readers convenience. To distinguish this more abstract discussion from the rest of the paper, and to highlight the special role of the sets we are about to introduce, we diverge from the convention of denoting sets via calligraphic capital letters and use blackboard bold capital letters.
	
	We start out be investigating a more general question. So, let $\mathbb{V}$ be a vector space of dimension $n$. For a (possibly nonconvex) cone $\mathbb{K}\subseteq \mathbb{V}$, and vectors $\Qb,\Hb_0\in\mathbb{V}$ and a convex set $\mathbb{J}\subseteq\mathrm{conv}(\mathbb{K})$. We want to know when we have the equality:
	\begin{align*}
		\min_{\Xb\in \mathbb{V}} \lrbr{\scp{\Qb}{\Xb} \colon \Xb\in\mathbb{K}\cap\mathbb{J},\ \scp{\Hb_0}{\Xb} = 1} =\min_{\Xb\in \mathbb{V}} \lrbr{\scp{\Qb}{\Xb} \colon \Xb\in\mathbb{J},\ \scp{\Hb_0}{\Xb} = 1}?
	\end{align*}  
	Defining $\mathbb{H}\coloneqq \lrbr{\Xb\colon \scp{\Hb_0}{\Xb} = 1 }$, we can equivalently ask for conditions for the equality 
	\begin{align*}
		\mathrm{conv}(\mathbb{H}\cap\mathbb{K}\cap\mathbb{J}) =\mathbb{H}\cap\mathbb{J}.
	\end{align*}
	The following theorem gives an answer based on convex geometry. 
	\begin{thm}\label{thm:KimsGeneralGeometricApproach}
		For $\mathbb{H},\mathbb{K},\mathbb{J}$ as above, assume that $\mathbb{H}\cap\mathbb{J}\neq \emptyset$ is bounded and that $\mathbb{J}$ is a face of $\mathrm{conv}(\mathbb{K})$. Then $\mathrm{conv}(\mathbb{H}\cap\mathbb{K}\cap\mathbb{J}) =\mathbb{H}\cap\mathbb{J}$.
	\end{thm}
	\begin{proof}
		For the "$\subseteq$"-inclusion, since $\mathbb{H}\cap\mathbb{K}\cap\mathbb{J}\subseteq \mathbb{H}\cap\mathbb{J}$ and the latter set is convex, there is nothing left to show.
		For the converse, let $\Xb\in \mathbb{H}\cap\mathbb{J}$. Then $\Xb\in \mathrm{conv}(\mathbb{K})$ since $\mathbb{J}\subseteq \mathrm{conv}(\mathbb{K})$, so that $\Xb = \sum_{i=1}^{n}\Xb_i$ with $ \Xb_i \in \mathbb{K}\setminus\lrbr{\Ob}$ but also $\Xb_i \in \mathbb{J}$ since $\mathbb{J}$ is a face of  $\mathrm{conv}(\mathbb{K})$ so that $\Xb_{i}\in \mathbb{K}\cap\mathbb{J}$.
		Now, $\scp{\Hb_0}{\Xb_{i}}>0$ since $\mathbb{H}\cap\mathbb{J}\neq \emptyset$  is bounded. Define $\gl_{i} = \scp{\Hb_0}{\Xb_{i}}.$ 
		We have $\scp{\Hb_0}{\Xb} = \sum_{i}\scp{\Hb_0}{\Xb_{i}} =  \sum_{i}\gl_{i} = 1$ and 
		$\gl_{i}^{-1}\Xb_{i}\eqqcolon \bar\Xb_{i} \in \mathbb{K}\cap\mathbb{J}$ and thus 
		$\Xb = \sum_{i}\gl_{i} \bar\Xb_{i} \in \mathrm{conv}(\mathbb{H}\cap\mathbb{K}\cap\mathbb{J})$.   
	\end{proof}
	This theorem motivates the search for a condition that lets us identify faces of convex cones, which are provided in the following theorem.
	\begin{thm}\label{thm:IdentifyFacesofK}
		Assume that $\mathbb{J} = \lrbr{\Xb\in\mathrm{conv}(\mathbb{K})\colon \scp{\Ab_i}{\Xb} = 0,\ i \in \irg{1}{m}}$ and define 
		$\mathbb{J}_p \coloneqq \lrbr{ \Xb\in\mathrm{conv}(\mathbb{K})\colon \scp{\Ab_i}{\Xb} = 0,\ i \in \irg{1}{p}}$ so that $\mathbb{J}_m=\mathbb{J}$ and  $\mathbb{J}_0=\mathrm{conv}(\mathbb{K})$. \linebreak
		If $\Ab_p \in \mathbb{J}_{p-1}^*,\ \blue{p} \in \irg{1}{m}$ then $\mathbb{J}$ is a face of $\mathrm{conv}(\mathbb{K})$. 
	\end{thm}
	\begin{proof}
		\blue{Since} a face of a face a convex set is itself a face of that set, the claim will follow by induction if we can show that 	$$\Ab_p \in \mathbb{J}_{p-1}^* \implies \mathbb{J}_p \mbox{ is a face of }\mathbb{J}_{p-1}.$$ So let $\mathbb{J}_p\ni\Xb = \Xb_1+\Xb_2$ with $\Xb_i\in \mathbb{J}_{p-1},\ i \in \lrbr{1,2}$. We have $\scp{\Ab_p}{\Xb_i}\geq0$ since  $\Ab_p \in \mathbb{J}_{p-1}^*$ so that $0 = \scp{\Ab_p}{\Xb} = \scp{\Ab_p}{\Xb_1}+\scp{\Ab_p}{\Xb_2}$ implies that actually $\scp{\Ab_p}{\Xb_i}=0$ and we indeed have $\Xb_i\in\mathbb{J}_p,\ i\in\lrbr{1,2}$. 
	\end{proof}
	
	Based on the above theorems, it is quite straight forward to prove the classical result from \cite{burer_copositive_2009}, at least for the case where the linear portion of the set is bounded, with $\mathbb{K} = \mathrm{ext}\CPP(\R_+\times\KK)$ and $\mathbb{J}$ equal to the feasible set of the conic reformulation (we omit laying out the details here, but the steps required are equivalent to the ones laid out in the proof of \cref{thm:DecomposableQCQPExactCMPReformulation}). A natural question is, whether we can execute a similar strategy for proving the exactness of a conic reformulation of reduced dimension by replacing the cone of extreme rays of $\CPP(\KK)$ with another appropriately structured object as our choice for $\mathbb{K}$.
	
	In the following theorem we show that by choosing $\mathbb{K}  =\mathrm{gen}\CMP\left(\left(\R_+\times\KK_0\right),\KK_1,\dots,\KK_S\right)$ and $\mathbb{J}$ and $\mathbb{H}$ appropriately we can use \cref{thm:KimsGeneralGeometricApproach} in order to obtain an exact conic reformulation of (\ref{eqn:DecomposableQuadraticProblem}).  
	
	\begin{thm}\label{thm:DecomposableQCQPExactCMPReformulation}
		Considering (\ref{eqn:DecomposableQuadraticProblem}), assume $\FF_i~\coloneqq~\lrbr{\left(\x\T,\y_i\T\right)\in\KK_0\times\KK_i\colon \Fb_i\x+\Gb_i\y_i = \r_i}$  are nonempty bounded sets. Further, assume that 
		\begin{align}\label{asmp:KeyAssumption}
			\begin{pmatrix}
				\x,\y_i
			\end{pmatrix}\in \FF_i,\ i \in \irg{1}{S} \implies Q_{j}(\x,\y_1,\dots,\y_S)\geq 0,\ j \in \irg{1}{K}. 
		\end{align}
		Then (\ref{eqn:DecomposableQuadraticProblem}) is equivalent to the following conic optimization problem: 
		\begin{align}\label{eqn:DecomposableQCQPCMPRefromulation}
			\begin{split}
				\min_{[\Xb]\in \SS_{n_1+n_2+1}^{S,n_1+1}} [\Cb]\odot[\Xb]& \\
				\mathrm{s.t.:} \ 
				[\Hb_0]\odot[\Xb] &= 1,\\
				[\Fb_i]\odot[\Xb] &= 0,\ i \in \irg{1}{S},\\
				\hat{Q}_{j}\left([\Xb]\right) &= 0, \ j \in \irg{1}{K}, \\
				[\Xb] &\in \CMP\left(\left(\R_+\times\KK_0\right),\KK_1,\dots,\KK_S\right),
			\end{split}
		\end{align}
		where $[\Cb],[\Hb_0],[\Fb_i]\in\SS_{n_1+n_2+1}^{S,n_1+1}, \ i\in\irg{1}{S}$ are defined as 
		\begin{align*}
			[\Cb]&\coloneqq \Gamma^{-1}
			\begin{pmatrix}
				0 & \tfrac{1}{2}\aa\T & \tfrac{1}{2}\cc_1\T & \dots & \tfrac{1}{2}\cc_S\T\\
				\tfrac{1}{2}\aa &\Ab & \tfrac{1}{2}\Bb_1 &\dots & \tfrac{1}{2}\Bb_S \\ 
				\tfrac{1}{2}\cc_1\T &\tfrac{1}{2}\Bb_1\T & \Cb_1 & \dots & \Ob\\
				\vdots&\vdots& \vdots & \ddots & \vdots \\
				\tfrac{1}{2}\cc_S\T &\tfrac{1}{2}\Bb_S\T & \Ob & \dots  & \Cb_S
			\end{pmatrix}, \quad [\Hb_0] = \Gamma^{-1}\left(\e_1\e_1\T\right),   \\ \\
			[\Fb_i] &\coloneqq \Gamma^{-1} \left(\left(-\r_i, \Fb_i,\Ob,\dots,\Gb_i,\dots,\Ob\right)\T\left(-\r_i, \Fb_i,\Ob,\dots,\Gb_i,\dots,\Ob\right)\right).
		\end{align*}
		and $\hat{Q}_{j}(\cdot)\colon\SS^{S,n_1+1}_{n_1+n_2+1} \rightarrow \R$ are linear functions such that 
		\begin{align*}
			\hat{Q}_{j}\left(\Gamma^{-1}\left(
			\begin{pmatrix}
				x_0\\\x\\\y_1\\\vdots\\\y_S
			\end{pmatrix}
			\begin{pmatrix}
				x_0\\\x\\\y_1\\\vdots\\\y_S
			\end{pmatrix}\T\right) \right) = Q_{j}(\x,\y_1,\dots,\y_S), \ j \in \irg{1}{K}
		\end{align*}
	\end{thm}
	\begin{proof}
		Consider the following equivalences 
		\begin{align*}
			\Fb_i\x+\Gb_i\y_i&= \r_i,\ \y_i\in\KK_i,\ i \in\irg{1}{S}, \ \x\in\KK_0,\\
			Q_{j}(\x,\y_1,\dots,\y_S)&=0, \ j \in \irg{1}{K},\\
			&\Updownarrow\\
			\begin{Vmatrix}
				\left(-\r_i, \Fb_i,\Ob,\dots,\Gb_i,\dots,\Ob\right)
				\begin{pmatrix}
					x_0\\\x\\\y_1\\\vdots\\\y_S
				\end{pmatrix}
			\end{Vmatrix}^2& = 0,\ \y_i\in\KK_i,\ i \in\irg{1}{S}, \ \x\in\KK_0,\ x_0\geq 0 ,\ x_0^2 = 1,\\
			Q_{j}(\x,\y_1,\dots,\y_S)&=0, \ j \in \irg{1}{K},\\
			&\Updownarrow\\
			[\Xb] = \Gamma^{-1}\left(
			\begin{pmatrix}
				x_0\\\x\\\y_1\\\vdots\\\y_S
			\end{pmatrix}
			\begin{pmatrix}
				x_0\\\x\\\y_1\\\vdots\\\y_S
			\end{pmatrix}\T\right),\ [\Fb_i]\odot[\Xb]& = 0,\ \y_i\in\KK_i,\ i \in\irg{1}{S}, \ \x\in\KK_0,\ x_0\geq 0 ,\ x_0^2 = 1,\\
			\hat{Q}_{j}([\Xb]) &= 0 \ j\in \irg{1}{K}\\
			&\Updownarrow\\
			[\Hb_0]\odot[\Xb] &= 1,\\
			[\Fb_i]\odot[\Xb] &= 0,\ i \in \irg{1}{S},\\
			\hat{Q}_{j}([\Xb]) &= 0 \ j\in \irg{1}{K}, \\
			[\Xb] &\in \mathrm{gen}\CMP\left(\left(\R_+\times\KK_0\right),\KK_1,\dots,\KK_S\right).
		\end{align*}
		Invoking \cref{thm:KimsGeneralGeometricApproach}, we specify 
		\begin{align*}
			\mathbb{K} &= \mathrm{gen}\CMP\left(\left(\R_+\times\KK_0\right),\KK_1,\dots,\KK_S\right), \\
			\mathbb{H} &= \lrbr{[\Xb]\in\SS_{n_1+n_2+1}^{S,n_1+1}\colon [\Hb_0]\odot[\Xb] = 1},
		\end{align*}
		and \blue{we need to show that }
		\begin{align*}
			\mathbb{J} &= \lrbr{[\Xb]\in\CMP\left(\left(\R_+\times\KK_0\right),\KK_1,\dots,\KK_S\right)\colon 
				\begin{array}{l}
					[\Fb_i]\odot[\Xb] = 0,\ i \in \irg{1}{S}\\
					\hat{Q}_{j}\left([\Xb]\right) = 0, \ j\in \irg{1}{K}
			\end{array}},
		\end{align*} is a face of $\CMP\left(\left(\R_+\times\KK_0\right),\KK_1,\dots,\KK_S\right)$. By \cref{thm:IdentifyFacesofK}, this will follow if we can show that $[\Fb_i]\odot[\Xb]\geq 0, \ \forall [\Xb]\in \CMP\left(\left(\R_+\times\KK_0\right),\KK_1,\dots,\KK_S\right),\ i \in \irg{1}{S}$ and that $\hat{Q}_{j}([\Xb])\geq 0,\ j \in \irg{1}{K}$ whenever $[\Xb]$ fulfills the homogeneous and conic constraints in the description of the feasible set of the conic optimization problem. We will first show, that the statement of the theorem \blue{would hold} if the quadratic constraints were omitted. Indeed for any of the $[\Fb_i]$ and any  $[\Xb]\in\mathrm{gen}\CMP\left(\left(\R_+\times\KK_0\right),\KK_1,\dots,\KK_S\right)$ we have 
		\begin{align*}
			\hspace{-1cm}
			\begin{split}
				[\Fb_i]\odot[\Xb] =& \left(\left(-\r_i, \Fb_i,\dots,\Gb_i,\dots\right)\T\left(-\r_i, \Fb_i,\dots,\Gb_i,\dots\right)\right)\bullet
				\begin{pmatrix}
					x_0^2 & x_0\x\T & x_0\y_1\T & \dots & x_0\y_S\T\\
					x_0\x&\x\x\T & \x\y_1\T &\dots & \x\y_S\T \\ 
					x_0\y_1&\y_1\x\T & \y_1\y_1\T  & \dots & \Ob \\
					\vdots&\vdots& \vdots & \ddots & \vdots \\
					x_0\y_S&\y_S\x\T & \Ob & \dots  & \y_S\y_S\T
				\end{pmatrix}\\
				=&\begin{pmatrix}
					-\r_i,\Fb_i,\Gb_i
				\end{pmatrix}\T\begin{pmatrix}
					-\r_i,\Fb_i,\Gb_i
				\end{pmatrix}\bullet
				\begin{pmatrix}
					x_0^2 & x_0\x\T & x_0\y_i\T\\
					x_0\x&\x\x\T & \x\y_i\T\\ 
					x_0\y_i&\y_i\x\T &\y_i\y_i\T
				\end{pmatrix}\\
				=& \begin{Vmatrix}
					\begin{pmatrix}
						-\r_i,\Fb_i,\Gb_i
					\end{pmatrix} 
					\begin{pmatrix}
						x_0\\ \x\\ \y_i
					\end{pmatrix}
				\end{Vmatrix}^2\geq 0.
			\end{split}
		\end{align*}
		To complete the first part of the argument we need to show that the feasible set is bounded. To this end we consider its recession cone which by \cite[Corollary 8.3.3.]{rockafellar_convex_2015} is  given by 
		\begin{align*}
			\hspace{-1cm}
			0^+\FF\coloneqq \lrbr{[\Xb] \in \CMP\left(\left(\R_+\times\KK_0\right),\KK_1,\dots,\KK_S\right) \colon 
				[\Hb_0]\odot[\Xb] = 0,\
				[\Fb_i]\odot[\Xb] = 0,\ i \in \irg{1}{S}
			}.
		\end{align*}
		Take an arbitrary $\left[\Xb\right]\in 0^+\FF$, then 
		\begin{align*}
			\left[\Fb_i\right]\odot\left[\Xb\right] &= \sum_{l=1}^{k}\gl_l
			\begin{Vmatrix}
				\begin{pmatrix}
					-\r_i,\Fb_i,\Gb_i
				\end{pmatrix} 
				\begin{pmatrix}
					x^0_l\\\x_l\\\y^i_l
				\end{pmatrix}
			\end{Vmatrix}
			^2 = 0, \ i \in \irg{1}{S} \mbox{ and }\\
			\left[\Hb_0\right]\odot\left[\Xb\right] &= \sum_{l=1}^{k} \left(x^0_l\right)^2 = 0, \mbox{ implying that } x^0_l = 0,\ l \in \irg{1}{k}. 
		\end{align*}
		Thus, for any $i \in \irg{1}{S}$ and $l\in\irg{1}{k}$ we have $\Fb_i\x_l+\Gb_i\y^i_l = \oo$ and $\left(0,\x_l,\y^i_l\right)\in \R_+\times\KK_0\times\KK_i$ so that we have a element of the recession cone of $\FF_i$, which only contains the origin by the boundedness assumption, so that $\left[\Xb\right] = \left[\Ob\right]$. So far our arguments imply that 
		\begin{align*}
			\hat{\mathbb{J}} \coloneqq \lrbr{[\Xb]\in\CMP\left(\left(\R_+\times\KK_0\right),\KK_1,\dots,\KK_S\right)\colon [\Fb_i]\odot[\Xb] = 0,\ i \in \irg{1}{S}}
		\end{align*}  
		is a face of $\mathbb{K}$, hence its extreme points correspond to extreme rays of $\mathbb{K}$ by \cref{thm:KimsGeneralGeometricApproach}, that is \linebreak  $\mathrm{gen}\CMP\left(\left(\R_+\times\KK_0\right),\KK_1,\dots,\KK_S\right)$. But then (\ref{asmp:KeyAssumption}) implies that  $\hat{Q}_{j}([\Xb])\geq 0,\ j\in \irg{1}{K}$ whenever $[\Xb]\in\hat{\mathbb{J}}$ so that by \cref{thm:IdentifyFacesofK} the set $\mathbb{J}$ is a face of $\mathbb{K}$ and our theorem follows from \cref{thm:KimsGeneralGeometricApproach}. 
	\end{proof}
	
	While the above representation of the conic problem is convenient for the application of Theorems \ref{thm:KimsGeneralGeometricApproach} and \ref{thm:IdentifyFacesofK} and the statement of the proof, we can use \cite[Proposition 3]{burer_copositive_2012} in order to present it in a more familiar form: 
	\begin{align}\label{eqn:CMPsimple}
		\begin{split}
			\min_{\Xb,\Yb_i,\Zb_i,\x,\y_i} \Ab_i\bullet\Xb + \aa\T\x&+\sum_{i=1}^{S} \left[\Bb_i\bullet\Zb_i+ \Cb_i\bullet\Yb_{i} +\cc_i\T\y\right]\\
			\mathrm{s.t.:}\  \Fb_i\x+\Gb_i\y_i &= \r_i, \quad \hspace{0.7cm} i\in\irg{1}{S},\\
			\diag\left(
			\begin{pmatrix}
				\Fb_i & \Gb_i
			\end{pmatrix}
			\begin{pmatrix}
				\Xb & \Zb_i\T\\\Zb_i & \Yb_i
			\end{pmatrix}
			\begin{pmatrix}
				\Fb_i\T\\ \Gb_i\T
			\end{pmatrix}\right) &= \r_i\circ\r_i, \quad i\in\irg{1}{S},\\
			\hat{Q}_{j}(\x,\Xb,\y_1,\Zb_1,\Yb_1,\dots,\y_S,\Zb_S,\Yb_S) &= 0, \quad \hspace{0.6cm} \ j\in\irg{1}{K},\\
		\left[
		\begin{pmatrix}
			1  & \x\T & \y_i\T\\
			\x &\Xb & \Zb_i\T \\ 
			\y_i &\Zb_i & \Yb_i
		\end{pmatrix}\right]_{i\in\irg{1}{S}} &\in \CMP\left(\left(\R_+\times\KK_0\right),\KK_1,\dots,\KK_S\right).
		\end{split}
	\end{align}

	Before discussing this new type of conic reformulation, we want to point out, that there is a another way to prove \cref{thm:DecomposableQCQPExactCMPReformulation}. First, we make the following observation:
	\begin{thm}\label{thm:CharacterizationOfCompletability}
		The partial matrix  
		\begin{align*}
			\Mb_{*}\coloneqq 
			\begin{pmatrix}
				\Xb & \Zb_1\T &\dots & \Zb_S \\ 
				\Zb_1 & \Yb_1  & \dots & * \\
				\vdots& \vdots & \ddots & \vdots \\
				\Zb_S & * & \dots  & \Yb_S
			\end{pmatrix},
		\end{align*}
		is completable to a matrix in $\CPP(\KK_0\times_{i=1}^S\KK_i)$ if and only if there are decompositions 
		\begin{align*}
			\begin{pmatrix}
				\Xb & \Zb_i\T \\ \Zb_i & \Yb_i
			\end{pmatrix} = 
			\begin{pmatrix}
				\bar{\Xb}\bar{\Xb}\T  & \bar{\Xb}\bar{\Yb}_i\T \\ \bar{\Yb}_i\bar{\Xb}\T & \bar{\Yb}_i\bar{\Yb}_i\T
			\end{pmatrix},  \mbox{ with }
			\begin{pmatrix}
				\bar{\Xb}\\	\bar{\Yb}_i
			\end{pmatrix} \in \KK_0^{r}\times\KK_i^{r} , \ i\in \irg{1}{S}, \ r\in\N,
		\end{align*}
		hence, if and only if 
		\begin{align*}
				\left[\begin{pmatrix}
				\Xb & \Zb_i\T \\ 
				\Zb_i & \Yb_i
			\end{pmatrix}\right]_{i\in\irg{1}{S}} &\in \CMP\left(\left(\R_+\times\KK_0\right),\KK_1,\dots,\KK_S\right).
		\end{align*}
	\end{thm}
	\begin{proof}
		Given said decompositions we can create a matrix 
		\begin{align*}
			\begin{pmatrix}
				\bar{\Xb} \\ \bar{\Yb}_1 \\ \vdots \\ \bar{\Yb}_S
			\end{pmatrix} \mbox{ for which }
			\begin{pmatrix}
				\bar{\Xb} \\ \bar{\Yb}_1 \\ \vdots \\ \bar{\Yb}_S
			\end{pmatrix}
			\begin{pmatrix}
				\bar{\Xb} \\ \bar{\Yb}_1 \\ \vdots \\ \bar{\Yb}_S
			\end{pmatrix}\T = 
			\begin{pmatrix}
				\bar{\Xb}\bar{\Xb}\T & \bar{\Xb}\bar{\Yb}_1\T& \dots & \bar{\Xb}\bar{\Yb}_S\T \\
				\bar{\Yb}_1\bar{\Xb}\T & \bar{\Yb}_1\bar{\Yb}_1\T& \dots &\bar{\Yb}_1\bar{\Yb}_S\T\\
				\vdots & \vdots & \ddots &\vdots\\
				\Yb_S\bar{\Xb}\T&\bar{\Yb}_S\bar{\Yb}_1\T&\dots& \bar{\Yb}_S\bar{\Yb}_S\T 
			\end{pmatrix}\in \CPP(\times_{i=0}^{S}\KK_i),
		\end{align*}
		is the desired completion of $\Mb_*$. Conversely, if $\Mb_{*}$ has a completion $\Mb\in\ \CPP(\times_{i=0}^{S}\KK_i)$ then by definition of the latter cone we have 
		\begin{align*}
			\Mb  = \begin{pmatrix}
				\bar{\Xb}\bar{\Xb}\T & \bar{\Xb}\bar{\Yb}_1\T& \dots & \bar{\Xb}\bar{\Yb}_S\T \\
				\bar{\Yb}_1\bar{\Xb}\T & \bar{\Yb}_1\bar{\Yb}_1\T& \dots &\bar{\Yb}_1\bar{\Yb}_S\T\\
				\vdots & \vdots & \ddots &\vdots\\
				\Yb_S\bar{\Xb}\T&\bar{\Yb}_S\bar{\Yb}_1\T&\dots& \bar{\Yb}_S\bar{\Yb}_S\T 
			\end{pmatrix},
		\end{align*}
		so  that 	
		\begin{align*}
			\begin{pmatrix}
				\Xb & \Zb_i\T \\ \Zb_i & \Yb_i
			\end{pmatrix} = 
			\begin{pmatrix}
				\bar{\Xb}\bar{\Xb}\T  & \bar{\Xb}\bar{\Yb}_i\T \\ \bar{\Yb}_i\bar{\Xb}\T & \bar{\Yb}_i\bar{\Yb}_i\T
			\end{pmatrix},  \mbox{ with }
			\begin{pmatrix}
				\bar{\Xb}\\	\bar{\Yb}_i
			\end{pmatrix} \left(\KK_0\times\KK_i\right)^r, \ i\in \irg{1}{S},\ r\in \N.
		\end{align*}
	\end{proof}
	\begin{rem}\label{rem:RemarkonCMP}
		The \blue{theorem} is easily derived, but it highlights the key difficulty for the construction of a completion of the arrow-head arrangement of a set of matrix blocks connected by a common submatrix $\Xb$. If all of the blocks have representations as convex-conic combinations (i.e.\ nonnegative linear combinations) where the parts of the representations that form the connecting $\Xb$-component are identical for all blocks, obtaining the completion is simply a matter of concatenating the individual factors of the decompositions. However, there is no guarantee that decompositions that are coordinated in this manner do exist. \blue{Nonetheless, we can now clearly see the correspondence between $\CMP$ and $\CPP$ hinted at, at the end of \cref{sec:The space of connected}: From any matrix in $\CPP$, if its ground cone is given by $\times_{i=0}^S\KK_i$,  one can carve out an arrow head shaped partial matrix $\Mb_*$ so that $\Gamma_{*}^{-1}(\Mb_*)$ is an element of $\CMP$. Conversely, such an element uniquely corresponds to a partial arrow head matrix that can be completed to at least one element of $\CPP$.} 	
	\end{rem}
	
	Now, it is clear that the following optimization problem is equivalent to (\ref{eqn:DecomposableQCQPBurer}): 
	\begin{align}
		\begin{split}
			\min_{\Xb,\Yb_i,\Zb_i,\x,\y_i} &\Ab_i\bullet\Xb + \aa\T\x+\sum_{i=1}^{S} \left[\Bb_i\bullet\Zb_i+ \Cb_i\bullet\Yb_{i} +\cc_i\T\y\right]\\
			\mathrm{s.t.:}\ &\mbox{ the linear constraints of (\ref{eqn:DecomposableQCQPBurer}) hold and } \\
			&\begin{pmatrix}
				1 & \x\T &\y_1\T & \dots & \y_S\T \\
				\x& \Xb & \Zb_1\T & \dots & \Zb_S\T \\
				\y_1&\Zb_1 & \Yb_{1} & \dots &* \\
				\vdots&\vdots&\vdots&\ddots & \vdots \\
				\y_S & \Zb_S & * &\dots & \Yb_{S}
			\end{pmatrix} \mbox{ can be completed to a matrix in } \CPP(\R_+\times_{i=0}^S\KK_i),
		\end{split}
	\end{align}
	but the latter constraint holds whenever the conic constraint in (\ref{eqn:CMPsimple}) holds. Thus, we can close the relaxation gap between (\ref{eqn:CMPsimple}) and (\ref{eqn:DecomposableQuadraticProblem}) by appealing to Burer's reformulation and \cref{thm:CharacterizationOfCompletability}. However, we believe it is valuable to have a direct proof that is solely based on the geometry of $\CMP$ and does not explicitly reference matrix completion. Firstly, we avoid referencing something abstract, namely completability, by invoking something relatively concrete, i.e.\ the geometry of the respective convex cone. Secondly, the proof shows that the homogenized feasible set of (\ref{eqn:CMPsimple}) is a face of the respective instance of $\CMP$, which may be a useful insight for future investigations of this object. Finally, the proof is a somewhat unexpected application of  the theory laid out in \cite{kim_geometrical_2020}, which may inspire similar approaches to convex reformulations where a desired property, in our case completability, is inscribed in the structure of the cone $\mathbb{K}$.  
	  
	To summarize, the reformulation we obtained is similar to the one obtainable from \cite{burer_copositive_2009} in that it is a linear-conic optimization problem over an appropriately structured convex cone. The advantage of our reformulation is that the number of variables is $S(n_1+n_2)(n_1+n_2+1)/2$, while for the traditional approach this number would be $(n_1+Sn_2)(n_1+Sn_2+1)/2$, which is a bigger number if $S$ is big enough. However, similarly to $\CPP$, we cannot directly optimize over $\CMP$ since no workable description is yet known for this novel object. We therefore propose the following strategy.
	
	\subsection{A new strategy for sparse conic reformulations}
	
	As stated before, optimizing over $\CMP$ necessitates the applications of appropriate inner and outer approximations of that cone. On the one hand, we thus look for necessary conditions $[\Mb]\in\SS^{S,k}_n$ has to meet lest completing $\Gamma_{*} \left([\Mb]\right)$ to a matrix in the respective $\CPP$-cone is impossible, and we denote the subset of connected components that meet these conditions by $\CC_{nes}\supseteq \CMP$. On the other hand we look for subsets $\CC_{suf}\subseteq \CMP$, in other words, we look for sufficient conditions on a connected component $[\Mb]\in\SS^{S,k}_n$ so that $\Gamma_{*} \left([\Mb]\right)$ is in fact completable. 
	
	As we will show in the next section such necessary and/or sufficient conditions can be formulated in terms of $\CPP$ constraints. Such constraints are again intractable in general so we need an additional step in order to take advantage of these approximations. Set-compeltely positive matrix cones are very well studied objects and strong inner and outer approximations feature prominently in the existing literature \blue{(see \cite{bomze_copositive_2012} for extensive discussion)}. These approximations can thus be used to find tractable approximations of $\CC_{suf}$ and $\CC_{nes}$. More precisely, whenever we describe $\CC_{nes}$ via set-completely postive constraints, we can loosen these constraints via tractable outer approximation of $\CPP$ as to obtain a new set $\CC_{outer}\supseteq\CC_{nes}$. Conversely, replacing $\CPP$ in the description $\CC_{suf}$ with a tractable inner approximation we obtain an inner approximation $\CC_{inner}\subseteq\CC_{suf}$. In total we get:
	
	$$\CC_{inner} \ \subseteq \ \CC_{suf} \ \subseteq \ \CMP \ \subseteq \CC_{nes} \ \subseteq \ \CC_{outer},$$
	hence, tractable inner and outer approximations of $\CMP$. 
	
	The two step nature of our proposed approximation procedure stems from the fact that there are two sources of difficulty that necessitate resorting to approximations. The first one is the requirement of completability, which is addressed by the inner two of the above inclusions. The second one is the requirement of set-completely positivity, addressed by the outer two of the above inclusions. 
	
	Hence, whenever we approximately solve (\ref{eqn:CMPsimple}) by replacing $\CMP$ by its tractable inner and outer approximations we incur a relaxation gap that consists of two components. The portion of the gap \blue{that} results from a failure of meeting the completability requirement we henceforth refer to as \textit{completability gap}, while the portion of the gap the stems from the approximation error caused by the relaxation of the $\CPP$ constraints will be refered to as the \textit{completepositivity gap}.   
	
	In the next section we will mostly be concerned with narrowing the completability gap by providing promising examples for $\CC_{nes}$ and $\CC_{suf}$. Also, most of the discussion in the rest of the article will focus on the quality of this gap. We will, however, also provide some references to approximations of $\CPP$, in order to give some orientation on how to narrow the completepositivity gap as well. In the section on our numerical experiments we will also show some strategies on how to bypass this gap entirely, albeit in limited cases.

	\section{Inner and  outer approximations of $\CMP$ based on set-completely positive matrix cones}\label{sec:Inner and  outer approximation of CMP based on set-completely positive matrix cones}
	
	Our goal in this section is to identify conditions on an element $[\Mb] \in \SS^{S,k}_n$ that are either sufficient or necessary for $\Gamma_*([\Mb])$ to have a set-completely positive completion. In the following discussion we will show that many such conditions can be given in terms of set-completely \blue{positive cone constraints.}   
	
	\subsection{An outer approximation via necessary conditions}

	For a vector of ground cones $\bar{\KK} \coloneqq \left(\KK_0,\dots,\KK_S\right)$, we define yet another generalization of the set-completely positive matrix cone 
	\begin{align*}
		\CPI\left(\bar{\KK}\right) &\coloneqq \lrbr{ 
			\left[
			\begin{pmatrix}
				\Xb & \Zb_1\T \\ \Zb_1 & \Yb_1
			\end{pmatrix},\dots,
			\begin{pmatrix}
				\Xb & \Zb_S\T \\ \Zb_S & \Yb_S
			\end{pmatrix} \right]
			\colon 
			\begin{pmatrix}
				\Xb & \Zb_i\T \\ \Zb_i & \Yb_i
			\end{pmatrix}\in \CPP(\KK_0\times\KK_i),\ i \in \irg{1}{S}
		},
	\end{align*}
	for which we can prove the following.

	\begin{thm}
		We have that $\CPI\left(\bar{\KK}\right)  \supseteq \CMP\left(\bar{\KK}\right)$.% is a convex superset. 
	\end{thm}
	\begin{proof}
%		We start out by showing that $\CPI$ is a convex cone. To see this consider that for the sum of  two members $[\bar{\Xb_1}],[\bar{\Xb_2}]\in\CPI$ we have 
%		\begin{align*}
%			[\bar{\Xb_1}]\oplus[\bar{\Xb_2}] = &\left[
%			\begin{pmatrix}
%				\Xb^1 & (\Yb^1_i)\T \\ \Yb^1_i & \Zb^1_i
%			\end{pmatrix}\right]_{\inirgs} 			
%			\oplus
%			\left[
%			\begin{pmatrix}
%				\Xb^2 & (\Yb^2_i)\T \\ \Yb^2_i & \Zb^2_i
%			\end{pmatrix}\right]_{\inirgs} 			\\
%			=& \left[
%			\begin{pmatrix}
%				\Xb^1+\Xb^2 & (\Zb^1_i+\Zb^2_i)\T \\ (\Zb^1_i+\Zb^2_i) & \Yb^1_i+\Yb^2_i
%			\end{pmatrix} \right]_{\inirgs},
%		\end{align*}
%		where 
%		\begin{align*}
%			\begin{pmatrix}
%				\Xb^1+\Xb^2 & (\Zb^1_i+\Zb^2_i)\T \\ (\Zb^1_i+\Zb^2_i) & \Yb^1_i+\Yb^2_i
%			\end{pmatrix}\in \CPP^n ,\ i\in\irg{1}{S},
%		\end{align*}
%		and the top left entries are identical as required. 
	By setting $\Xb = \x\x\T,\ \Zb_i = \y_i\x,\ \Yb_i = \y_i\y_i\T$ we see that the generators of $\CMP$ are contained in $\CPI$ and by convexity $\CMP$ itself is contained.
	\end{proof}
		
	We thus have an outer approximations of $\CMP$ in terms of set-completely positive matrix blocks, which is convenient for approximately optimizing of over $\CMP$ since set-completely positive optimization is a well researched field.

	\subsection{Inner approximations via sufficient conditions}
	
	We define 
	\begin{align*}
		\CPS\left(\bar{\KK}\right) &\coloneqq \lrbr{ 
			\left[
			\begin{pmatrix}
				\Xb & \Zb_i\T \\ \Zb_i & \Yb_i
			\end{pmatrix} \right]_{\inirgs}
			\colon 
			\begin{pmatrix}
				\Wb_i & \Zb_i\T \\ \Zb_i & \Yb_i
			\end{pmatrix}\in \CPP(\KK_0\times\KK_i),\ i \in \irg{1}{S},\ \sum_{i=1}^S \Wb_i=  \Xb 
		}.
	\end{align*}
	While it is not immediately obvious, the above cone is in fact a subset of $\CMP$. In fact the generators of $\CPS$ are a subset of the generators of $\CMP$ as we will now show. 
	
	\begin{thm}\label{thm:CPS}
		$\CPS\left(\bar{\KK}\right)\subseteq\CMP\left(\bar{\KK}\right)$ if $\KK_i, \ i \in \irg{1}{S}$ contain the origin. 
	\end{thm}
	\begin{proof}
	 Let $[\bar{\Xb}] \in \CPS$, we need to  to show that 
		\begin{align*}
			\begin{pmatrix}
				\Xb & \Zb_i\T \\ \Zb_i & \Yb_i
			\end{pmatrix} = 
			\sum_{k=1}^r 
			\begin{pmatrix}
				\x^k\\ \y_i^k
			\end{pmatrix}
			\begin{pmatrix}
				\x^k\\ \y_i^k
			\end{pmatrix}\T, \ \mbox{with }
			\begin{pmatrix}
				\x^k\\ \y_i^k
			\end{pmatrix}\in\KK_0\times\KK_i, \ k \in \irg{1}{r},\ \inirgs.
		\end{align*}
		for some fixed $r\in\N$. The important aspect is that the decomposition of the $\Xb$-component does not change across $\inirgs$. We have  
		\begin{align*}
			\begin{pmatrix}
				\Wb_i & \Zb_i\T \\ \Zb_i & \Yb_i
			\end{pmatrix} = 
			\sum_{k=1}^{r_i} 
			\begin{pmatrix}
				\w_i^k\\ \y_i^k
			\end{pmatrix}
			\begin{pmatrix}
				\w_i^k\\ \y_i^k
			\end{pmatrix}\T \mbox{with }
			\begin{pmatrix}
				\w^k_i\\ \y_i^k
			\end{pmatrix}\in\KK_0\times\KK_i,\ k \in \irg{1}{r_i}, \ \inirgs.
		\end{align*}
		We can set $r = \sum_{i=1}^Sr_i$ and we have $\Xb = \sum_{i=1}^S\Wb_i = \sum_{i=1}^S\sum_{k=1}^{r_i} \w_i^k(\w_i^k)\T$ so that 
		\begin{align*}
			\begin{pmatrix}
				\Xb & \Zb_i\T \\ \Zb_i & \Yb_i
			\end{pmatrix} = 
			\sum_{k=1}^{r_i} 
			\begin{pmatrix}
				\w_i^k\\ \y_i^k
			\end{pmatrix}
			\begin{pmatrix}
				\w_i^k\\ \y_i^k
			\end{pmatrix}\T +
			\sum_{j\in \irg{1}{S}\setminus\lrbr{i}}
			\sum_{k=1}^{r_j} 
			\begin{pmatrix}
				\w_j^k\\ \oo
			\end{pmatrix}
			\begin{pmatrix}
				\w_j^k\\ \oo
			\end{pmatrix}\T.		
		\end{align*}
		with 
		\begin{align*}
			\begin{pmatrix}
				\w_i^k\\ \y_i^k
			\end{pmatrix}\in\KK_0\times\KK_i,\ k \in \irg{1}{r_i},\ 
			\begin{pmatrix}
				\w_j^k\\ \oo
			\end{pmatrix}\in\KK_0\times\KK_i,\ k \in \irg{1}{r_j}
			,\ j \in \irg{1}{S}\setminus\lrbr{i},
		\end{align*}
		where the last inclusion holds, since $\KK_i$ contain the origin. 
	\end{proof}
	
	For obtaining a second approximation, we can use a slight generalization of known results on matrix completion to obtain another inner approximation for the case $\KK_0=\R^{n_1}_+$. We define 
	\begin{align*}
		\hspace{-1cm}\mathcal{CBC}_k \left(\bar{\KK}\right) &\coloneqq \lrbr{ 
			\left[
			\begin{pmatrix}
				\Xb & \Zb_i\T \\ \Zb_i & \Yb_i
			\end{pmatrix} \right]_{\inirgs}
			\colon 
			\begin{pmatrix}
			x & \z_i\T \\ \z_i & \Yb_i
			\end{pmatrix}\in \CPP(\R_+\times\KK_i), 
			\begin{array}{l}
				\Zb_i = \z_i\e_k\T\\
				\Xb  = x\e_k\e_k\T 
			\end{array},
			\ i \in \irg{1}{S}\
		},
	\end{align*}
	and 
	\begin{align*}
		\mathcal{CBC}\left(\bar{\KK}\right) \coloneqq \sum_{k=1}^{n_1}\mathcal{CBC}_k \left(\bar{\KK}\right).
	\end{align*}
	Then we can prove the containment 
	\begin{thm}
		$\mathcal{CBC}\left(\bar{\KK}\right)\subseteq \CMP\left(\bar{\KK}\right)$ for $\KK_0=\R^{n_1}_+$.
	\end{thm}
	\begin{proof}
		Since convex cones are closed under addition the statement will follow if we show that 
		$\mathcal{CBC}_k\left(\bar{\KK}\right)\subseteq \CMP\left(\bar{\KK}\right)$ for any $k\in\irg{1}{n_1}$. For an element of $\CBC_k$ consider the partial matrix 
		\begin{align*}
			\Mb\coloneqq\begin{pmatrix}
				x & \z_1\T &\dots & \z_S\T \\ 
				\z_1 & \Yb_1  & \dots & * \\
				\vdots& \vdots & \ddots & \vdots \\
				\z_S & *  & \dots  & \Yb_S
			\end{pmatrix}, \mbox{for which by construction }
			\begin{pmatrix}
				x & \z_i\T \\ \z_i & \Yb_i  
			\end{pmatrix} \in \CPP\left(\R_+\times \KK_i\right) \mbox{ holds.}
		\end{align*}
	\blue{
	If $x=0$ then it follows that $\z_i = \oo, \  \inirgs$ in which case a completion of $\Mb$ to a matrix in $\CPP(\R_+\times_{i=1}^S\KK_i)$ is easily constructed by concatenating the zero vector with the decompositions of $\Yb_i= \bar{\Yb}_i\bar{\Yb}_i\T, \ \bar{\Yb}_i\in\KK_i^{r_i}, \ \inirgs$, where we can insert columns of zeros in case $r_i$ are not all identical.} Thus, we can assume $x=1$. We will proof that $\Mb$ can \blue{still} be completed to a member in $\CPP(\R_+\times_{i=1}^S\KK_i)$. The desired inclusion then follows since zero rows and columns can be added in order to obtain a member of $\CPP(\R_+^{n_1}\times_{i=1}^S\KK_i)$. 
	
	Our proof involves merely a slight adaptation of the argument used for the completion of partial completely positive matrices given in \cite{drew_completely_1998}, who considered the case where $\KK_i$ are all positive orthants. We show that such an assumption is unnecessary. Let us proceed by induction and start by showing that the first $(2n_2+1)\times (2n_2+1)$ principal submatrix of $\Mb$ can be completed to a matrix in $\CPP\left(\R_+\times\KK_1\times\KK_2\right)$. After a \blue{permutation}, this matrix can be written as
	\begin{align*}
		\bar{\Mb}\coloneqq\begin{pmatrix}
			\Yb_1 & \z_1  & \Xb\T \\
			\z_1\T & 1 & \z_2\T \\
			\Xb & \z_2  & \Yb_2 
		\end{pmatrix} 
	\end{align*}
	where we replaced the unspecified entries by $\Xb$. Observe that the submatrices 
	\begin{align*}
		\Mb_1 \coloneqq\begin{pmatrix}
			\Yb_1 & \z_1 \\ \z_1\T & 1
		\end{pmatrix}\in\CPP\left(\KK_1\times\R_+\right), \
		\Mb_2 \coloneqq\begin{pmatrix}
			1 & \z_2\T \\ \z_2 & \Yb_2
		\end{pmatrix}\in\CPP\left(\R_+\times\KK_2\right),
	\end{align*} 
	so that 
	\begin{align*}
		\Mb_1 = \sum_{l=1}^{m_1}
		\begin{pmatrix}
			\f_l \\ f^0_l
		\end{pmatrix}
		\begin{pmatrix}
			\f_l \\ f^0_l
		\end{pmatrix}\T \mbox{ with } 
	\begin{pmatrix}
		\f_i \\ f^0_i
	\end{pmatrix}\in \KK_1 \times \R_+,\\
	\Mb_2 = \sum_{k=1}^{m_2}
	\begin{pmatrix}
	  	g^0_k \\ \g_k
	\end{pmatrix}
	\begin{pmatrix}
		g^0_k \\ \g_k
	\end{pmatrix}\T \mbox{ with } 
	\begin{pmatrix}
		g^0_k \\ \g_k
	\end{pmatrix}\in\R_+\times \KK_2 .\\
	\end{align*}
	\blue{Let us define $m_1m_2$ vectors as follows:}
	\begin{align}
		\v_{lk} \coloneqq
		\begin{pmatrix}
			 g^0_k\f_l \\
			f^0_lg^0_k\\
			f^0_l\g_k
		\end{pmatrix} \in \KK_1\times \R_+ \times\KK_2, \ l \in\irg{1}{m_1}, \ k \in\irg{1}{m_2}.
	\end{align}
	Then the matrix $\sum_{k,l}\v_{lk}\v_{lk}\T$ is the matrix $\bar{\Mb}$ with $\Xb =\z_2\z_1\T$. Hence, after undoing the perturbation, we generate the desired completion. For the $j$-th induction step we can repeat the argument with $\KK_1$ replaced by $\times_{i=1}^{j-1}\KK_1$ and $\KK_2$ replaced by $\KK_j$. 
	\end{proof}
	
	Finally, we present a simple, yet, as we will see in the numerical experiments, very effective inner approximation, which is applicable whenever $\KK_i\in\lrbr{\R^{n_2}_+,\ \R^{n_2}}, \ \inirgs$. Again, we express it as the sum of simpler cones given by 
	\blue{ 
		\begin{align*}
		\hspace{-1cm}\mathcal{DDC}_{k,s} \left(\bar{\KK}\right) &\coloneqq \lrbr{ 
			\left[
			\begin{pmatrix}
				\Xb & \Zb_i\T \\ \Zb_i & \Yb_i
			\end{pmatrix} \right]_{\inirgs}
			\colon 
			\begin{pmatrix}
				\Xb & \z \\ \z\T & y
			\end{pmatrix}\in \CPP(\KK_0\times\R_+), 
			\begin{array}{l}
				\Zb_s = \z\e_k\T,\\
				\Yb_s = y\e_k\e_k\T,\\
				\Yb_i = \Ob,\  \inirgs\setminus\lrbr{s},\\
				\Zb_i = \Ob,\  \inirgs\setminus\lrbr{s}
			\end{array}
		}.
	\end{align*}
	}
	We can then define  
	\begin{align*}
		\mathcal{DDC} \left(\bar{\KK}\right) \coloneqq \sum_{s=1}^{S}\sum_{k=1}^{n_2}\mathcal{DDC}_{k,s} \left(\bar{\KK}\right),
	\end{align*}
	about which the following statement is easily proved. 
	
	\begin{thm}
		We have $\mathcal{DDC} \left(\bar{\KK}\right) \subseteq \CMP \left(\bar{\KK}\right)$ if $\KK_i\in\lrbr{\R^{n_2}_+,\ \R^{n_2}}, \ \inirgs$.
	\end{thm}
	\begin{proof}
	 	Since $\CMP$ is convex, it is enough to proof that $\mathcal{DDC}_{s,k} \left(\bar{\KK}\right) \subseteq \CMP \left(\bar{\KK}\right)$ for any $k\in\irg{1}{n_2}, \ s \in\irg{1}{S}$. For any $[\Mb]\in \mathcal{DDC}_{s,k} \left(\bar{\KK}\right)$ the required completion of $\Gamma_{*}([\Mb])$ is easily obtained by filling out the unspecified entries with zeros. 
	\end{proof}
	Note, that the statement remains true if we merely work with a selection of $\mathcal{DDC}_{k,s}$ in order to alleviate some \blue{of} the numerical burden. 
	
	All these inner and outer approximations we now discussed represent an effort to tackle the completability gap. But, as we laid out at the beginning of this section, they all have it in common that they are constructed using set-completely positive matrix cones, over which we cannot optimize directly. In this text we will discuss some instances where the completepositivity gap can be bypassed conveniently, so that we can focus on assessing the extent of the completability gap.

	\subsubsection{Limitations of the inner approximations of $\CMP$}\label{sec:Limitations of the inner approximations}
	We will now critically asses the strength of the inner approximations discussed above. Of course, an obvious limitation of $\mathcal{CBC}$ and $\mathcal{DDC}$  is that the $\Xb$ and the $\Yb_i$ components respectively can only be diagonal matrices. In case of $\mathcal{CBC}$, this has some undesirable consequences when approximating an exact reformulation of (\ref{eqn:DecomposableQuadraticProblem}) based on \cref{thm:DecomposableQCQPExactCMPReformulation}. 
	
	Obviously, if $\CMP$ is replaced by $\mathcal{CBC}$, then $\x = \oo$, since these values reside in off-diagonal of the north-west blocks. But then \cref{thm:DecomposableQCQPExactCMPReformulation} implies that $\x$ is the convex combination of some $\x_j,\ j \in \irg{1}{k}$ that are part of a feasible solution to (\ref{eqn:DecomposableQuadraticProblem}). Since the feasible set is bounded we get $\x_j= \oo$ as well, which eventually yields  $\Zb_i =  \sum_{j=1}^{k}\gl_j\y^i_j\x_j\T = \Ob$, so that the approximations eliminates these components entirely.  
	
	A similar deficiency can be identified for $\CPS$. To see this, note the from the proof of \cref{thm:CPS} we have that all extreme rays of $\CPS$ are in fact rank one, in the sense that all matrix components are rank one matrices. In other words the generators of $\CPS$ are a subset of the generators of $\CMP$. Hence, if $\CMP$ is replaced by $\CPS$ in (\ref{eqn:DecomposableQCQPCMPRefromulation}) the extreme points of the feasible \blue{set} can be shown to be rank one as well, by invoking a similar argument as in \cref{thm:DecomposableQCQPExactCMPReformulation}. If $\Yb_i,\y_i, \Wb_i, \ \inirgs$ are part of feasible extremal solution of the respective approximation we get
	\begin{align*}
		\begin{pmatrix}
			w_0^i  & \y_i\T \\ \y_i & \Yb_i
		\end{pmatrix}\in\SS^{n_2+1}_+, \ \inirgs, \quad \sum_{i=1}^{S}w^i_0 = 1,
	\end{align*}  
	where $w_0^i$ are the north-west entries of $\Wb_i,\ \inirgs$. By Schur complementation we get $\SS^{n_\blue{2}}_+ \ni w_0^i\Yb_i-\y_i\y_i\T = w_0^i\y_i\y_i\T-\y_i\y_i\T$, for any fixed $\inirgs$, which implies that either $w^0_i = 1$ or $\y_i = \oo$. Thus, the approximations based on $\CPS$ eliminates all but one of the $\Yb_i$ components and as a consequence all but one of the $\Zb_i$ components. Depending on the model at hand, this can be an advantage as we will see in the numerical experiments in \cref{sec:Numerical experiments}. However, in case $(\Fb_i,\ \Gb_i, \r_i)$ is identical across $\inirgs$, the approximations actually eliminates all $\Yb_i$ and $\Zb_i$ components. To see this, consider that in said case we have that $\diag(\Gb_i\Yb_i\Gb_i\T) = \Ob$ for one $i$ forces the same for all $i$, which by boundedness implies $\Yb_i = \Ob$, entailing $\Zb_i = \Ob$ for all $\inirgs$.  
	
	Despite these limitations we have found instances of (\ref{eqn:DecomposableQuadraticProblem}) where the inner approximations yield favorable results. We will discuss these instances in the next section, where we conduct numerical experiments assessing the efficacy of the inner and outer approximations.

	\section{Numerical experiments}\label{sec:Numerical experiments}
	
	As discussed in the introduction, the authors of \cite{bomze_two-stage_2022} tried to solve (\ref{eqn:St3QP}) using copositive reformulations, where they compared the traditional model akin to (\ref{eqn:DecomposableQCQPBurer}), with what we can now conceptualize as the $\CPI$ relaxation of the $\CMP$ reformulation of (\ref{eqn:St3QP}). For both models, they used their respective $\DNN$ relaxations in order to produce solutions.  For the purpose of certifying optimality, they exploited the fact that either relaxation also produces feasible solutions, hence upper bounds, since both leave the original space of variables in tact. For many instances, these bounds alone closed the optimality gap, but for some gaps persisted, even though they were narrowed by extensive polishing procedures, about which we will not go into detail here. What we are setting out to do in this section is revisiting these instances and new variants of them, in order to see if the bounds we introduced in this article can further narrow the optimality gap. 
	
	In what follows we will use Mosek as a conic optimization solver, and Gurobi as a global optimization solver, to both of which we interface via the YALMIP environment in Matlab (see \cite{Lofberg2004}). All experiments were run on a Intel Core i5-9300H CPU with 2.40GHz and 16GB of ram.

	In our epxeriments consider the following problem %generalizations of (\ref{eqn:St3QP})
	\begin{align}\label{eqn:St3QPgen} 
		v(\FF) \coloneqq \min_{\x\in\R^{n_1},\y_i\in\R^{n_2}} \lrbr{\x\T\Ab\x+ \sum_{i=1}^{S}p_i \left(\x\T\Bb_i\y_i + \y_i\T\Cb_i\y_i\right) \colon (\x,\bar{\y})\in \FF,},	
	\end{align}
 with the following specifications for $\FF$: 
	\begin{align*}
		\FF_1 &\coloneqq \lrbr{ 
			\begin{pmatrix}
				\x \\ \bar{\y}
			\end{pmatrix}\in \R^{n_1+Sn_2}_+ \colon \e\T\x+\e\T\y_i = 1,\ \inirgs
		}, \\
%		\FF_2 &\coloneqq \lrbr{ 
%		\begin{pmatrix}
%			\x \\ \bar{\y}
%		\end{pmatrix}\in \R^{n_1}\times\R^{Sn_2}_+ \colon \e\T\x+\e\T\y_i = 1,\ \inirgs, \ \|\x\|\leq 1 
%		 },\\
	 \FF_2 &\coloneqq \lrbr{ 
	 	\begin{pmatrix}
	 		\x \\ \bar{\y}
	 	\end{pmatrix}\in \lrbr{0,1}^{S}\times\R^{Sn_2} \colon \e\T\x = (S-1),\ \sum_{i=1}^{S}\y_i\T\y_i = 1,\ \y_ix_i = \oo,\ \inirgs 
	  },\\
 	 \FF_3 &\coloneqq \lrbr{ 
 		\begin{pmatrix}
 			\x \\ \bar{\y}
 		\end{pmatrix}\in \R^{n_1}_{\blue{+}}\times\R^{Sn_2} \colon \x\T\x \blue{+} \sum_{i=1}^{S}\y_i\T\y_i = 1 
 	}.
	\end{align*}
		where $\bar{\y} \coloneqq \left(\y_1,\dots,\y_S\right)$.
		
	The data for the objective functions coefficients were generated using the same two approaches as in \cite{bomze_two-stage_2022}. Next to setting $p_i = 1/S, \ \inirgs$, the following two schemes for generating the problem data have been implemented:
	
	\begin{itemize}
		\item[Scheme 1:] For the first one, we sample  $n_1+n_2$ points from the unit square. The first $n_1$ points are fixed and their mutual distances are used to populate the entries in $\Ab$. For the other $n_2$ points we assume that they are only known to lie in square with side length $2\eps$, where their position follows a uniform distribution. For these points $S$ samples are generated and for the $s$-th sample, the distances between them and the first $n_1$ points populate the entries of $\Cb_s$ and $\Bb_s$ respectively.  
		
		\item[Scheme 2:] For the second one, we choose $A_{ij} \sim \mathcal{U}_{\lrbr{0,1}}$, $B_{ij}\sim\mathcal{U}_{[0:10]}$, $C_{ij} \sim \mathcal{U}_{[0,0.1]}$, independently of each other, where $\mathcal{U}_{\mathcal{M}}$ is the uniform distribution with support  $\mathcal{M}$.
	\end{itemize}
	\blue{It was observed in \cite{bomze_two-stage_2022} that Scheme 2 consistently produced instances where the gap generated via the $\CPI$-approximation was large. Note, that in the experiments there, the authors focused exclusively on $\FF_1$. }
	
	We will now proceed with a discussion of the different instances of $\FF_i, \ i \in \irg{1}{\blue{3}}$, where we present the respective conic reformulations/relaxations and the inner and outer approximations of its sparse counterpart. Regarding the inner approximations, note that \blue{one could combine them by using Minkowski sums of the different cones. However,  for a given problem there will} only ever be one non-redundant approximation. The reason is that the linear functions attain the optimum at an extreme point of the feasible set and the extreme rays of a sum of cones are a subset of the extreme rays of the individual cones. Thus, for every $\FF_i, \ i \in \irg{1}{\blue{3}}$ we will discuss the merits of only one specific inner approximation at a time. The \blue{lower bounds} will be obtained by using $\CPI$ by default. The focus of the discussion will be the quality of the bounds obtained. Specifically, we are interested in assessing the completability gap, which necessitates guaranteeing a completepositivity gap of zero. We will discuss how the latter was achieved case by case.    
	
	\subsection{Using $\mathcal{DDC}$ under $\FF_1$}
	By choosing $\FF= \FF_1$ we are recovering the scenario problem for the two-stage stochastic standard quadratic optimization problem introduced in \cite{bomze_two-stage_2022}. In the experiments conducted there, a conic lower bound was used that is equivalent to the outer approximation of (\ref{eqn:DecomposableQCQPCMPRefromulation}) based on $\CPI$. Since the original space of variables is preserved, the conic relaxation also yielded an upper bound that conveniently closed the optimality gap for all instances generated by sc heme 1. However, the gaps generated by the $\CPI$-approximation were typically large. %The lower bounds however, were numerically almost identical to the ones obtainable from the traditional approach based on (\ref{eqn:DecomposableQCQPBurer}), which is known to perform well.
	%The authors were able to narrow the gap substantially using the upper bound to warm start local and global optimization procedures. 
	In this section we will test whether the gap can be also be improved by using the inner approximations introduced here. 
	
	Due to the limitations discussed in \cref{sec:Limitations of the inner approximations}, the only inner approximation that is meaningfully applicable here is the one based on $\mathcal{DDC}$. Thus, the approximation will involve $Sn_2$ constraints involving $\CPP(\R^{n_1+1}_+\times\R_+)$. In case $n_1+2\leq 4$ these constraints can be represented via semidefinite constraints, since 
	$\CPP(\R^n_+) = \SS^n_+\cap \NN^n\eqqcolon \DNN^n$ whenever $n\leq 4$, so that the completepositivity gap can be conveniently bypassed. If $n_1+2>4$ the relaxation based on $\DNN$ is an outer approximation of $\mathcal{DDC}$, which itself is an inner approximation of $\CMP$ so that we cannot qualify the resulting approximation as neither outer nor inner. However, the original space of variables stays in tact regardless so that in cases where $n_1+2>4$ we can still obtain another upper bound that potentially narrows the optimality gap.

	\subsection{Using $\CPS$ under $\FF_2$}
	
	The model encodes selecting one out of $S$ groups of variables to be nonzero and optimizing the objective using just these variables. \blue{The activation and deactivation of the different groups is modeled via the variable $\x$, so that $n_1 = S$ in this model.} While there are more straightforward ways of encoding this process, the one presented here is the one for which the conic bounds behaved most favorably.  
	 
	In order to obtain a $\CMP$ reformulation for computing $v(\FF_2)$ via \cref{thm:DecomposableQCQPExactCMPReformulation} we would have to do some prior adaption of the problem. First of all, $x_i\in \{0,1\}, \inirgs$ can be reformulated as quadratic constraints $x^2_i-x_i =0$, so that in order for the assumptions of the theorem to hold, we would have to introduce redundant constraints and additional variables given by $x_i+s_i = 1,\ s_i\geq0, \ \inirgs$. Secondly, in order for $\y_ix_i=0$ to fulfill said assumptions we would have to split each $\y_i$ into a positive and a negative component and enforce the constraints for both components. Lastly, the quadratic constraints would need to be absorbed into the a second order cone constraint.  Due to the introduction of this many variables, we would have no chance at bypassing the completepositivity gap. Thus, we will merely work with the  the following $\CMP$ based relaxation: 
	\begin{align}\label{eqn:CMPRelaxationofF3}
		\begin{split}
			\min_{\Xb,\Yb_i,\Zb_i,\x,\y_i} \Ab_i\bullet\Xb &+\sum_{i=1}^{S} \blue{p_i} \left[\Bb_i\bullet\Zb_i+ \Cb_i\bullet\Yb_{i} \right]\\
			\mathrm{s.t.:}\ \e\T\x & = (S-1), \\
			\e\e\T\bullet\Xb & = (S-1)^2, \\
			\blue{\diag\left(\Xb\right)}& =  \x,\\
			\sum_{i=1}^{S}\blue{\Ib\bullet}\Yb_i &= 1, \\
			\Zb_i\e_{\blue{i}}& = 0, \ \inirgs, \\ 
			\left[
			\begin{pmatrix}
				1  & \x\T & \y_i\T\\
				\x &\Xb & \Zb_i\T \\ 
				\y_i &\Zb_i & \Yb_i
			\end{pmatrix}\right]_{i\in\irg{1}{S}} &\in \CMP\left(\R_+^{n_1+1},\R^{n_2},\dots, \R^{n_2}\right).
		\end{split}
	\end{align}
	
	When working with the $\CPS$ based \blue{upper} bound, we obtain a problem with $S$ conic constraints involving $\CPP(\R_+^{n_1+1}\times\R^{n_2})$. Here we can use a result from \cite[Theorem 1]{natarajan_reduced_2017}, which states that  
	\begin{align}\label{eqn:Natarajan}
		\CPP(\R_+^{n_1+1}\times\R^{n_2}) = 
		\lrbr{
		\begin{pmatrix}
			\Mb_1  & \Mb_2\T \\ \Mb_2 & \Mb3
		\end{pmatrix}\in \SS^{n_1+n2+1}_{\blue{+}}\colon \Mb_1 \in \CPP(\R^{n_1+1}_+) 
		}.
	\end{align}
	This allows us to bypass the completepositivity gap whenever $n_1+1 \leq 4$.

	\blue{
	However, while the above problem is clearly a lower bound, we cannot proof that it is tight based on the theory we have discussed in this text. Thus, when we calculate the optimal values obtained based on inner and outer approximations of $\CMP$ we merely bound the optimal value of the the $\CMP$ relaxation. Nonetheless, the space of original variables $\x$ and $\y_i, \ \inirgs$ stays in tact for any of these approximations so that we can obtain upper bounds to the original problem and hence a valid optimality gap. It is however not as straight forward as in the previous model since the values of  $\x$ and $\y_i, \ \inirgs$ obtained from approximations of  the $\CMP$ relaxation do not necessarily fulfill the nonlinear constraints in $\FF_2$, as their counterparts in the relaxation are only imposed on the lifted variables. 
	
	We therefore have to employ some rounding in order to obtain feasible solutions. In order for $\x$ to be feasible for $\FF_2$ all entries have to be equal to one except for a single one, say the $j$-the entry, that is equal to zero. Also, all $\y_i, \ \inirgs\setminus\lrbr{j}$ are zero so that $\y_j\T\y_j = 1$. For a solution obtained from a relaxation we therefore round the smallest entry of $\x$, again say the $j$-th entry, down to zero, while the rest is rounded up to one. Similarly, all $\y_i$ are set to zero except for $\y_j$. We obtain a feasible value for this variable by dividing it by its norm so that eventually $\y_j\T\y_j = 1$ holds.
	
	We also want to point out that in this specific case the rounding procedure can actually improve the upper bound obtained from $\x$ and $\y_i, \ \inirgs$ compared to their infeasible pre-rounding values. This is counter intuitive at first, since usually, when we take optimal solutions of a relaxation, we have to sacrifice some performance in order to turn them into feasible solutions of the original problem. However, this intuition leads us astray in this instance. Remember that $\x$ and $\y_i, \ \inirgs$ do not appear in the objective of the $\CMP$ relaxation and its approximations. They are just some values that are needed in order to make the optimal choices one the other variables feasible and that do not have any specific relation with the optimal value of the approximation itself beyond that. Thus, changing these variables may take their implied objective function value of the original problem in either direction. 
	}

	\subsection{Using $\CBC$ under $\FF_3$}
	In order to bypass the weakness of $\CBC$ outlined in \cref{sec:Limitations of the inner approximations} we will work with a simplified, sparse, conic reformulation given by 
	\begin{align}\label{eqn:F3CBC}
		\begin{split}
			\min_{\Xb,\Yb_i,\Zb_i,\x,\y_i} \Ab_i\bullet\Xb +\sum_{i=1}^{S} \blue{p_i} \left[\Bb_i\bullet\Zb_i\right.&\left.+ \Cb_i\bullet\Yb_{i}\right]\\
			\mathrm{s.t.:}\
			 \Ib\bullet\Xb + \sum_{i=1}^{S}\Ib\bullet\Yb_i & = 1,\\
			\left[
			\begin{pmatrix}
				\Xb & \Zb_i\T \\ 
				\Zb_i & \Yb_i
			\end{pmatrix}\right]_{i\in\irg{1}{S}} &\in \CMP\left(\R^{n_1}_+,\R^{n_2},\dots, \R^{n_2}\right).
		\end{split}
	\end{align}
	\blue{Note, that we cannot apply \cref{thm:DecomposableQCQPExactCMPReformulation} directly here since the single quadratic constraint does not fulfill the assumption of the theorem as it can, after the constant is put on the left-hand side, take both positive and negative values over the remaining feasible set. However,} the fact that this is in fact a valid reformulation and not just a lower bound can be deduced from \cref{thm:KimsGeneralGeometricApproach} by choosing $\mathbb{H}$ to be the hyperplane corresponding to the one linear constraint that is present in (\ref{eqn:F3CBC}), and $\mathbb{J}$ to be all of $\CMP$. The boundedness of $\mathbb{J}\cap\mathbb{H}$ follows from the fact that the identity matrix $\Ib$ is positive definite. In this reformulation $\x$ is absent so that the problem lined out in \cref{sec:Limitations of the inner approximations} is mute. Of course, this comes at the cost, of having merely a single upper bound given by the optimal solution of the $\mathcal{CBC}$ approximation. However, since $\KK_i= \R^{n_2}, \ \inirgs$ we can use the fact that $\CPP(\R_+\times\R^n) = \SS^n_+$ (see \cite[Section 2]{bomze_interplay_2021}) in order to close the completepositivity gap regardless of the dimension of the problem. We also like to note, that \blue{in our experiments, the lower bound was often very close to zero}, which leads to optimality gaps being \blue{ reported as $\infty$}. In order to avoid this inconvenience we added 1 as a constant to the objective function.

	\subsection{Design of the experiments and results}
	\blue{
	Based on the models discussed above we conducted three experiments. In the first one, we wanted to asses the quality of the bounds obtained from our approximations. This was done by evaluating the global optimality gaps obtained from these bounds, but also by comparing these gaps to the global optimality gaps that a benchmark solver, namely Gurobi, can produce in a reasonable amount of time. Gurobi is a commercial solver that employs branch and bound and related strategies to solve QCQPs globally, which makes it a good benchmark for our procedures, as they produce bounds on the global solution as well.  
	
	As it turned out that the conic bounds can be calculated very time efficiently, we conducted a second experiment where we set the time limit for Gurobi to the time it took for the conic approach to produce the respective optimality gaps. This allows us to assess a potential gain in efficiency for global solvers, if they employed the bounds derived in this text as a pre-solving step. 
	
	Finally, we compared the sparse bounds with the bounds obtained from relaxations of the full model, i.e. the model where the full $\CPP$ constraint was present, as in (\ref{eqn:DecomposableQCQPBurer}), rather than its sparse counterpart based on $\CMP$. In \cite{bomze_two-stage_2022}, these experiments have been conducted for instances of $\FF_1$, where the authors observed close to no gap between the full models and the sparse models. We repeat these experiments for instances of $\FF_2$ and $\FF_3$ to test whether this phenomenon persists, but also to test whether the advantage of the sparse models with respect to computation time is also maintained for the inner approximations. }

	\subsection{Quality of the bounds}
	
	For each of the models we generated two types of \blue{instances}. For the first one we choose the dimension of the problem such, that the completepositivity gap could be bypassed and one where that is not the case. For the latter instances, we worked with outer approximations of the respective set-completely positive constraints. Hence, the $\CPI$ relaxation was further relaxed, so that  the resulting problem can be qualified as a valid lower bound. The relaxation of the inner approximation does not allow for such a qualification, since we obtain lower bound to an upper bound. However, the relaxation yields valid upperbounds as a byproduct since the original space of variables stays in tact for all but the $\CBC$ \blue{approximation} of $\FF_3$. However, for the latter the completepositivity gap can be bypassed regardless of the dimension of the problem data. For every choice on $(n_1,n_2,S)$ and $\F_i, \ i = 1,2,3$ we generated 10 instances from scheme 1 and 2 respectively. For every instance we calculated the $\CPI$ lower bound, the bounds and approximations based on the respective inner approximations, and in addition we used upper and lower bounds achieved by Gurobi within a 5 minute time limit. 
	
	The results are summarized in \cref{tbl:GapsTable}. The "instance-types" are indicated by a quadruple of the form $n_1\_ n_2\_S\_ s$, where $s\in\lrbr{1,2}$ indicates the scheme by which we constructed the instances. In the multi-column  "Conic Gaps" we report the average gap between the $\CPI$ lower bound and the feasible solution generated from the $\CPI$ bound (UB), the optimal value of the inner approximations (I) and the feasible solution generated from the latter approximations (IUB). \blue{Note, that we also considered instances of sizes for which we could not guarantee that the completepositivity gap is eliminated. In these cases the bounds based on the  inner approximations are not valid upper bounds, since they stem from relaxations of inner approximations. For these cases we still report the gaps, but they appear in the table in parenthesis. However, we do like to mention at this point that these "invalid upper bounds" never fell below any of the lower bounds we calculated, which suggests that the completepositivity gap is small at least for our experiments.}  For "Gurobi Gaps" we calculate these gaps with respect to the lower bound fond by Gurobi instead of the one obtained from $\CPI$. In addition we present the gap between the $\CPI$ based lower bound and the upperbound generated by Gurobi (O), and we also report the optimality gap obtained by Gurobi itself within the 5 minute time limit (G). All the gaps are reported in percentages relative to the respective lower bound. Finally, in the last two multi-columns, we count the number of times the conic and the Gurobi gaps were smaller then $0.01\%$, at which point we consider the instance solved.        
	
	For the experiments on $\FF_1$ we see that \blue{the} phenomenon already documented in \cite{bomze_two-stage_2022} persists: instances from scheme 1 are regularly solved via $\CPI$ alone, while that is not the case for the scheme 2 instances. However, for these instances the solutions from the $\mathcal{DDC}$ approximation yield excellent bounds, revealing that both approximations are very good, albeit not quite good enough to solve the instances on a $0.01\%$ tolerance threshold. \blue{The feasible solutions produced by via $\mathcal{DDC}$ are only slightly better, than the ones produced by $\CPI$, and sometimes even worse.} We also note that $\CPI$ on average yields a much better lower bound than Gurobi does within the time limit, sometimes even certifying optimality of Gurobi's feasible solution. The upper bound provided by $\mathcal{DDC}$ performs \blue{worse} to Gurobi's upperbound, when measured relative to Gurobi's lower bound.    
	 
	Regarding $\FF_2$, we see that the conic gaps were narrowed quite substantially by the inner approximation and regularly closed. Gurobi on its own performed similarly except for the largest instance types, for which it was outperformed quite substantially. What is remarkable is the fact that the upper bounds of the approximations of the $\CMP$ relaxation seem to also upper bound Gurobi's lower bounds. Conversely, Gurobi's upper bounds seem to live close to the $\CPI$ based lower bounds on average. This might indicate that the $\CMP$ relaxation may in fact be tight despite the fact that \cref{thm:DecomposableQCQPExactCMPReformulation} is not applicable. We hope we can address this phenomenon in future research. Again, instances from scheme 2 seemed to be a greater challenge. Note that the smallest gaps are regularly produced by the feasible solution of the inner approximation.  
	
	Finally we can see that for $\FF_3$ the upper bound based on $\mathcal{CBC}$, unfortunately, performed quite poorly, which is surprising, given that the derivation of $\mathcal{CBC}$ is the one that is closest to classical results in matrix completion. On the brighter side, we see that the $\CPI$ produced good lower bounds that narrow the gap to Gurobi's feasible solution better than Gurobi itself.  
	
	We also recorded the average time spent on the different approaches in \cref{tbl:TimesTable}. We decomposed these running times into the time the respective solver used to produce the bounds (solver time), the internal model-building time reported by Yalmip (yalmip time) and the time our implementation used for building the model that is passed to Yalmip (model time). We like to point out a couple of things. Firstly, for the majority of the instance types Gurobi ran into the time limit on average. Also, on average the $\mathcal{DDC}$ approximations are more demanding for Mosek than the other approximations. Still, the models were solved quite quickly, certainly quicker than 5 minutes. Hence, our methods can produce good bounds with reasonable effort. Finally we would like to point out that there are also spikes in the model time for some of the inner approximations. This, however, is an artifact of our implementation that can potentially be avoided with better programming.   
	\blue{
	\subsection{Restricting Gurobi's time limit to the conic solution time} \label{sec:Restricting Gurobi's time limit to the conic solution time}
	
	For this experiment we used the same setup as in the previous section, but this time Gurobi's time limit was set to the time it took the conic solver to produce both the solutions of the inner and the outer approximations. This was done on a per instance basis, but the average solution times can be gathered from \cref{tbl:TimesTable}. The results are summarized in \cref{tbl:TimeLimit}, where the nomenclature is as in the previous section.  
	
	\begin{table}[htbp]
		\begin{center}
			\begin{tabular}{rl|ccc|cc}
				&\multirow{2}{*}{\textbf{Instance-types}}	& \multicolumn{3}{c|}{\textbf{Conic Gaps}} & \multicolumn{2}{c}{\textbf{Gurobi Gaps}} \\
				&  &  {UB} & {I} & {IUB} & {G} & {O} \\ 
				\hline
				\hline
				\multirow{8}{*}{$\FF_1$}
				&10\_10\_10\_1 & \textbf{0,00}  & (2,57)  & 2,19  & 71,01  & 16,60 \\ 
				&20\_20\_20\_1 & \textbf{0,00}  & (3,42)  & 3,21  & 175,97 & 37,68 \\ 
				&2\_10\_10\_1  & \textbf{0,00}  & 18,36 & 17,66 & 176,46 & 18,25 \\ 
				&2\_5\_5\_1    & \textbf{0,00}  & 16,97 & 16,21 & 33,80  & 4,36 \\ \cline{2-7}
				&10\_10\_10\_2 & 39,57 & ({0,29})  & \textbf{36,92} & 185,61 & 53,12 \\ 
				&20\_20\_20\_2 & 67,37 & ({0,32})  & \textbf{46,35} & 359,74 & 67,32 \\ 
				&2\_10\_10\_2  & 1,66  & \textbf{0,25}  & 23,42 & 100,60 & 25,00 \\ 
				&2\_5\_5\_2    & 2,53  & \textbf{0,18}  & 18,67 & 77,26  & 16,29 \\ 
				\hline
				\hline
				\multirow{4}{*}{$\FF_2$}
				&3\_10\_3\_1 & \textbf{0,00} & \textbf{0,00} & \textbf{0,00} & 4,43$^*$ & 0,44$^*$ \\ 
				&3\_5\_3\_1 & \textbf{0,00} & \textbf{0,00} & \textbf{0,00} & \textbf{0,00} & \textbf{0,00} \\  \cline{2-7}
				&3\_10\_3\_2 & 2,66 & 2,66 & 2,66 & \textbf{2,17}$^*$ & 2,70$^*$ \\ 
				&3\_5\_3\_2 & 1,21 & 1,21 & 1,21 & \textbf{0,00} & 1,21 \\ 
			\end{tabular}
			\caption{Comparison with Gurobi under a time limit}
			\label{tbl:TimeLimit}
		\end{center}
	\end{table}
	
	The first thing, that stands out is that we omitted results on $\FF_3$ and on the larger instances of $\FF_2$. For these instances Gurobi was not able to obtain optimality gaps within the time limit because it could not find a feasible solution and/or was not able to produce a lower bound. The issue was not resolved even after the time limit was expanded by several seconds. For the gaps marked with an asterisk in \cref{tbl:TimeLimit}, the same issue was present, but was resolved by expanding the time limit by an additional second. Of course this biases the respective results in favor of Gurobi, but we decided to include these results anyway, since the question what a slight loosening of the time limit could achieve is also interesting. 
	
	From the results we did include we see several things. First, we observe that the gaps produced via the conic approach were often much better than what Gurobi could achieve in the same time. Only in some of the instances of $\FF_2$ did Gurobi show an advantage, but only in cases where Gurobi operated under an extended time limit. Moreover, the feasible solutions generated by the conic approach were of better quality on average than the ones produced by Gurobi. It is therefore highly plausible, that a conic pre-solver could benefit the running time of global solvers such as Gurobi.

	\subsection{Comparing the sparse and the full conic models}\label{sec:Comparing the sparse and the full conic models}
	
	In our final experiments, we calculated the lower bounds obtained from relaxations of the full model (\ref{eqn:DecomposableQCQPBurer}), where the entire completely positive constraint was present, rather than its sparse surrogate. In cases where the original space of variables was preserved, we also calculated the associated upper bounds. For models $\FF_2$ and $\FF_3$ we considered the full model to be given by (\ref{eqn:CMPRelaxationofF3}) and (\ref{eqn:F3CBC}) where the $\CMP$ constraint was replaced by the respective full $\CPP$ constraint. These conic constraints were approximated using \cite[Theorem 1]{natarajan_reduced_2017} as depicted in (\ref{eqn:Natarajan}), where the completely positive constraints on the submatrices were relaxed to doubly nonnegative constraints. The results are summarized in \cref{tbl:FullComparison}. 
	
	\begin{table}[htbp]
		\begin{center}
			\begin{tabular}{rl|cccc|cc}
				&\multirow{2}{*}{\textbf{Instance-types}}	& \multicolumn{4}{c|}{\textbf{Gaps}}  & \multicolumn{2}{c}{\textbf{\# Gaps}}  \\
				& & M & F & UB & I &  UB &  I \\ \hline \hline
				\multirow{6}{*}{$\FF_1$}
				&10\_10\_10\_1 & 0,00 & 0,00 & 0,00 & (2,57) & 10 & (0) \\ 
				&2\_10\_10\_1 & 0,00 & 0,00 & 0,00 & 18,36 & 10 & 0  \\ 
				&2\_5\_5\_1 & 0,00 & 0,00 & 0,00 & 16,97 & 10 & 0  \\ \cline{2-8}
				&10\_10\_10\_2 & 0,04 & 22,12 & 27,52 & (0,31) & 0 & (6) \\ 
				&2\_10\_10\_2 & 0,00 & 0,04 & 1,66 & 0,25 & 1 & 0  \\ 
				&2\_5\_5\_2 & 0,00 & 1,88 & 2,53 & 0,18 & 2 & 1  \\ 
				\hline \hline
				\multirow{6}{*}{$\FF_2$}
				&10\_10\_10\_1 & 0,00 & 0,11 & 0,01 & (0,01) & 8 & (8) \\ 
				&3\_10\_3\_1 & 0,00 & 0,36 & 0,00 & 0,00 & 9 & 9  \\ 
				&3\_5\_3\_1 & 0,00 & 5,48 & 0,00 & 0,00 & 6 & 6  \\ \cline{2-8}
				&10\_10\_10\_2 & 0,00 & 1,77 & 1,34 & (1,21) & 9 & (1)  \\ 
				&3\_10\_3\_2 & 0,00 & 5,73 & 2,66 & 2,66 & 8 & 8  \\ 
				&3\_5\_3\_2 & 0,00 & 2,73 & 1,21 & 1,21 & 9 & 9  \\ 
				\hline
				\hline
				\multirow{6}{*}{$\FF_3$}&10\_10\_10\_1 & 0,00 & - & - & - & - & -  \\ 
				&10\_3\_10\_1 & 0,00 & - & - & - & - & - \\ 
				&5\_3\_5\_1 & 0,00 & - & - & - & - & -  \\  \cline{2-8}
				&10\_10\_10\_2 & 0,00 & - & - & - & - & -  \\ 
				&10\_3\_10\_2 & 0,00 & - & - & - & - & -  \\ 
				&5\_3\_5\_2 & 0,00 & - & - & - & - & -\\ 
			\end{tabular}
			\caption{Comparison with the full model}
			\label{tbl:FullComparison}
		\end{center}
	\end{table}

	In the first column (M) we report the gap between the lower bound from the full model and the lower bound based on $\CPI$. The second one (F) gives the gap between the lower bound of the full model and the upper bound associated with the feasible solution that the full model produced. The other two columns (UB, I) are calculated as in the other experiments. The final two columns report the number of times the optimality gap of the full model was worse than the one produced by either $\CPI$ alone (UB) or by $\CPI$ and the respective inner approximation (I). 
	
	We see that there is almost no difference between the lower bounds generated by the full and the sparse models. This is an observation already made for $\FF_1$ in \cite{bomze_two-stage_2022}, and it is replicated here for $\FF_2$ and $\FF_3$ as well. However, for the upper bounds the picture is more complicated. For instances of $\FF_1$ generated with scheme 1 the upper bounds from the full model and $\CPI$ are similar while the upper bounds from the inner approximation lack behind. But the later at least sometimes produces superior upperbounds when it comes to instances produced by scheme 2. For instances of $\FF_2$ the upperbounds produced by the sparse models are often better than what the full model is capable of producing. For $\FF_3$ no upperbounds were produced by the full model since the space of the original variables is absent. Thus, we also omitted the other upper bounds since there was nothing to compare them to.  
	
	As in \cite{bomze_two-stage_2022}, we again see that the bounds based on $\CPI$ can be computed much faster than the bounds from the full model. This holds true for the newly tested instances of $\FF_2$ and $\FF_3$ as well. On top of that, the same holds true for the inner approximations to the point where it is almost always faster to evaluate both the inner and the outer sparse approximations than it is to solve the full model. Disadvantages exist for the overall model building time, but this is a matter of implementation that is avoidable, at least in principle, by better implementations.   
	
	\begin{table}[htbp]
		\begin{tabular}{rl|rrr|rrr|rrr}
			&\multirow{2}{*}{\textbf{Instance-types}} & \multicolumn{3}{|c}{\textbf{Solver Time}} & \multicolumn{3}{|c}{\textbf{Yalmip Time}} & \multicolumn{3}{|c}{\textbf{Model Time}} \\
			&  & Full & Inner & Outer & Full & Inner & Outer & Full & Inner & Outer \\ \hline \hline
			\multirow{6}{*}{$\FF_2$}&10\_10\_10\_1 & 17,973 & 0,141 & 0,157 & 0,100 & 0,106 & 0,090 & 0,098 & 0,166 & 0,089 \\ 
			&3\_10\_3\_1 & 0,051 & 0,012 & 0,012 & 0,080 & 0,080 & 0,084 & 0,027 & 0,042 & 0,025 \\ 
			&3\_5\_3\_1 & 0,011 & 0,006 & 0,006 & 0,077 & 0,078 & 0,077 & 0,026 & 0,041 & 0,025 \\ \cline{2-11}
			&10\_10\_10\_2 & 41,611 & 0,130 & 0,296 & 0,088 & 0,092 & 0,085 & 0,081 & 0,140 & 0,082 \\ 
			&3\_10\_3\_2 & 0,188 & 0,014 & 0,044 & 0,077 & 0,078 & 0,077 & 0,026 & 0,042 & 0,024 \\ 
			&3\_5\_3\_2 & 0,025 & 0,007 & 0,013 & 0,077 & 0,085 & 0,077 & 0,025 & 0,041 & 0,024 \\ 
			\hline
			\hline
			\multirow{6}{*}{$\FF_3$}&10\_10\_10\_1 & 18,368 & 0,170 & 0,093 & 0,112 & 0,190 & 0,102 & 0,090 & 0,765 & 0,064 \\ 
			&10\_3\_10\_1 & 0,170 & 0,062 & 0,027 & 0,114 & 0,327 & 0,147 & 0,069 & 0,631 & 0,061 \\ 
			&5\_3\_5\_1 & 0,016 & 0,010 & 0,008 & 0,102 & 0,117 & 0,103 & 0,037 & 0,146 & 0,030 \\ \cline{2-11}
			&10\_10\_10\_2 & 14,511 & 0,161 & 0,089 & 0,091 & 0,156 & 0,085 & 0,066 & 0,542 & 0,051 \\ 
			&10\_3\_10\_2 & 0,161 & 0,049 & 0,031 & 0,105 & 0,171 & 0,103 & 0,061 & 0,541 & 0,059 \\ 
			&5\_3\_5\_2 & 0,017 & 0,012 & 0,008 & 0,107 & 0,116 & 0,107 & 0,035 & 0,143 & 0,028 \\ 
		\end{tabular}
		\caption{Running times including the full model }
		\label{}
	\end{table}
	
}
	
	\section*{Conclusion}
	In this text we presented a new approach to sparse conic optimization based on a generalization of the set-completely positive matrix cone, which was motivated by the study of the two-stage stochastic standard quadratic optimization problem. Using innner and outer approximations of said cone allows for certificates of exactness of a sparsification outside of traditional matrix completion approaches. We demonstrate in numerical experiments, that this approach can close or at least narrow the optimality gap in interesting cases. We think that this provides a \blue{proof} of concept that may motivate future research. Interesting questions remain, for example, about the quality of the inner approximations and whether they can be proven to be exact for special cases. 
	 
	\subsubsection*{Data availability statement}
	
	The datasets generated and/or analyzed during the current study are available from the corresponding author on reasonable request.
	
\begin{landscape}
\begin{table}[htbp]
	\begin{tabular}{ll|rrr|rrrrr|cccc|ccc}
		& \multirow{2}{*}{\textbf{Instance-types}} & \multicolumn{3}{c|}{\textbf{Conic Gaps}} & \multicolumn{5}{c|}{\textbf{Gurobi Gaps}} & \multicolumn{4}{c|}{\textbf{\# Gurobi Gaps}} & \multicolumn{3}{c}{\textbf{\#  Conic Gaps}}\\
		&  & \multicolumn{1}{c}{UB} & \multicolumn{1}{c}{I} & \multicolumn{1}{c|}{IUB} & \multicolumn{1}{c}{G} & \multicolumn{1}{c}{UB} & \multicolumn{1}{c}{I} & \multicolumn{1}{c}{IUB} & \multicolumn{1}{c|}{O} & UB & I & IUB & O & UB & I & IUB \\ 
		\hline
		\hline
		\multirow{8}{*}{$\FF_1$}
		&10\_10\_10\_1 & \textbf{0,00} & (2,57)  & 2,19    & 27,20  & 25,59 & (28,93)  & 28,44  & \textbf{1,27}  & 0 & 0 & (0) & 0 & \textbf{10} & (0) & 0 \\ 
		&20\_20\_20\_1 & \textbf{0,00} & (3,42)  & 3,21    & 164,45 & 95,96 & (102,58) & 102,19 & \textbf{35,32} & 0 & 0 & (0) & 0 & \textbf{10} & (0) & 0 \\ 
		&2\_10\_10\_1  & \textbf{0,00} & 18,36 & 17,66   & 23,16  & 21,63 & 44,12  & 43,31  & \textbf{1,24}  & 0 & 0 & 0 & 0 & \textbf{10} & 0 & 0 \\ 
		&2\_5\_5\_1    & \textbf{0,00} & 16,97 & 16,21   & 6,68   & 6,58  & 24,76  & 23,97  & \textbf{0,09}  & 0 & 0 & 0 & \textbf{1} & \textbf{10} & 0 & 0 \\ 
		\cline{2-17}
		&10\_10\_10\_2 & 39,57 & (0,29) & \textbf{36,92}   & 40,36  & 94,88  & (40,36) & 91,35  & \textbf{0,30}    & 0 & 0 & (0) & 0     & 0 & (0) & 0 \\ 
		&20\_20\_20\_2 & 67,37 & (0,32) & \textbf{46,35}   & 117,23 & 216,92 & (89,96) & 177,13 & \textbf{14,72}   & 0 & 0 & (0) & 0     & 0 & (0) & 0 \\ 
		&2\_10\_10\_2  & 1,66  & \textbf{0,25} & 23,42   & 0,48   & 2,15   & 0,73  & 24,04  & \textbf{0,00}    & 0 & 0 & 0 & \textbf{10}    & \textbf{1} & 0 & 0 \\ 
		&2\_5\_5\_2    & 2,53  & \textbf{0,18} & 18,67   & 0,08   & 2,61   & 0,25  & 18,77  & \textbf{0,00}    & 0 & 0 & 1 & \textbf{10}    & \textbf{1} & 0 & \textbf{1} \\ 
		\hline
		\hline
		\multirow{8}{*}{$\FF_2$}& 
		 10\_10\_10\_1  & \textbf{0,01} & ({0,01}) & \textbf{0,01}    & \textbf{0,01} & \textbf{0,01} & (0,01) & \textbf{0,01} & \textbf{0,01}    & 6 & 0 & (6) & \textbf{8}    & \textbf{8} & (8) & \textbf{8} \\ 
		 &20\_20\_20\_1 & \textbf{0,00} & ({0,00}) & \textbf{0,00}    & 4,53 & 4,35 & (4,35) & 4,35 & 0,17    & 0 & 0 & (0) & 0    & \textbf{9} & (10)& \textbf{9} \\ 
		 &3\_10\_3\_1   & \textbf{0,00} & \textbf{0,00} & \textbf{0,00}    & 0,01 & 0,01 & 0,01   & 0,01 & \textbf{0,00}    & 7 & 7 & 7 & \textbf{10}   & 10& 10& 10 \\ 
		 &3\_5\_3\_1    & \textbf{0,00} & \textbf{0,00} & \textbf{0,00}    & \textbf{0,00} & \textbf{0,00} & 0,01   & \textbf{0,00} & \textbf{0,00}    & \textbf{10}& 9 & \textbf{10}& \textbf{10}   & \textbf{10}& 9 & \textbf{10} \\ 
		\cline{2-17}
		&10\_10\_10\_2 & 1,63 & (1,48) & \textbf{1,48}    & \textbf{0,02}  & 0,16  & (0,01)  & \textbf{0,02}  & 1,48    & 1  & \textbf{5}  & (6)  & 2    & \textbf{2} & (2) & \textbf{2 }\\ 
		&20\_20\_20\_2 & 0,53 & (0,48) & \textbf{ 0,51}    & 12,20 & 11,09 & (11,03) & 11,07 & \textbf{1,54}    & 0  & 0  & (0)  & 0    & \textbf{3} & (3) & \textbf{3} \\ 
		&3\_10\_3\_2   & \textbf{2,66} & \textbf{2,66} & \textbf{2,66}    & \textbf{0,01}  & \textbf{0,01}  & \textbf{0,01}  & \textbf{0,01}    & 2,66    & 8  & \textbf{10} & 7    & 3    & \textbf{3} & \textbf{3}   & \textbf{3} \\ 
		&3\_5\_3\_2    & \textbf{1,21} & \textbf{1,21} &\textbf{ 1,21}    & \textbf{0,00}  & \textbf{0,00}  & \textbf{0,00}  & \textbf{0,00}    & 1,21    & \textbf{10} & \textbf{10} & \textbf{10}   & 6    & \textbf{6} & \textbf{6}   & \textbf{6} \\ 
		\hline
		\hline
		\multirow{8}{*}{$\FF_3$}
		& 10\_10\_10\_1 & - & 404,20& - & 1,97 & - & 412,79 & - & \textbf{0,27} & - & 0 & - & 0 & -&0 &- \\ 
		& 20\_20\_20\_1 & - & 484,57 & - & 1302,01 & - & 954,89 & - & \textbf{675,72} & - & 0 & - & 0 & - & 0 & - \\ 
		& 10\_3\_10\_1 & - & 907,85 & - & 0,62 & - & 914,15 & - & \textbf{0,00} & - & 0 & - & \textbf{10} & -&0 & -\\ 
		& 5\_3\_5\_1 & - & 257,07 & - & 0,21 & - & 257,83 & - & \textbf{0,00} & - & 0 & - & \textbf{10} & - & 0 & -\\ 
		\cline{2-17}
		& 10\_10\_10\_2 & - & 149,54 & - & 1,79 & - & 153,92 & - & \textbf{0,04} & - & 0 & - & 0 & -&0 & -\\ 
		& 20\_20\_20\_2 & - & 278,50 & - & 1909,44 & - & 304,09 & - &\textbf{ 1781,02} & - & 0 & - & 0 & -&0 & -\\ 
		& 10\_3\_10\_2 & - & 164,45 & - & 0,62 & - & 166,06 & - & \textbf{0,01} & - & 0 & - & \textbf{5} & -&0 &- \\ 
		& 5\_3\_5\_2 & - & 60,58 & - & 0,22 & - & 60,93 & - & \textbf{0,00} & - & 0 & - & \textbf{10} & -&0 & -\\ 
	\end{tabular}
	\caption{Results on the quality of bounds}
	\label{tbl:GapsTable}
\end{table}
\end{landscape}

\begin{landscape}
\begin{table}[htbp]
	\begin{tabular}{ll|rrr|rrr|rrr|rrr}
		& \multirow{2}{*}{\textbf{Instance-types}}  & \multicolumn{3}{c|}{\textbf{Solver time}}   & \multicolumn{3}{c|} {\textbf{Yalmip time}}   & \multicolumn{3}{c|}{\textbf{Model time}} &  \multicolumn{3}{c}{\textbf{Total} }  \\ 
		&  & \multicolumn{1}{c}{Exact} & \multicolumn{1}{c}{Inner} & \multicolumn{1}{c|}{Outer} & \multicolumn{1}{c}{Exact} & \multicolumn{1}{c}{Inner} & \multicolumn{1}{c|}{Outer} & \multicolumn{1}{c}{Exact} & \multicolumn{1}{c}{Inner} & \multicolumn{1}{c|}{Outer} & \multicolumn{1}{c}{Exact} & \multicolumn{1}{c}{Inner} &\multicolumn{1}{c}{Outer} \\ 
		\hline
		\hline
		\multirow{8}{*}{$\FF_1$}& 10\_10\_10\_1 & 300,088 & 1,269 & 0,149 & 0,126 & 0,280 & 0,132 & 0,229 & 1,419 & 0,176 & 300,443 & 2,969&0,457 \\ 
		& 20\_20\_20\_1 & 300,076 & 54,426 & 4,898 & 0,191 & 1,880 & 0,262 & 2,550 & 27,699 & 0,968 & 302,817 & 84,004&6,128 \\ 
		& 2\_10\_10\_1 & 300,339 & 0,102 & 0,036 & 0,105 & 0,176 & 0,108 & 0,097 & 0,638 & 0,116 & 300,540 & 0,916&0,259 \\ 
		& 2\_5\_5\_1 & 300,630 & 0,026 & 0,010 & 0,084 & 0,100 & 0,088 & 0,029 & 0,152 & 0,050 & 300,742 & 0,278&0,147 \\ 
		\cline{2-14}
		& 10\_10\_10\_2 & 300,052 & 1,093 & 0,159 & 0,098 & 0,228 & 0,106 & 0,176 & 1,072 & 0,148 & 300,326 & 2,394&0,412 \\ 
		& 20\_20\_20\_2 & 300,169 & 72,279 & 4,340 & 0,035 & 1,685 & 0,199 & 2,265 & 25,130 & 0,806 & 302,469 & 99,094&5,346 \\ 
		& 2\_10\_10\_2 & 300,767 & 0,103 & 0,043 & 0,102 & 0,203 & 0,111 & 0,118 & 0,710 & 0,130 & 300,987 & 1,016&0,283 \\ 
		& 2\_5\_5\_2 & 260,024 & 0,022 & 0,011 & 0,100 & 0,116 & 0,101 & 0,046 & 0,173 & 0,055 & 260,170 & 0,311&0,167 \\ 
		\hline
		\hline
		\multirow{8}{*}{$\FF_2$}& 10\_10\_10\_1 & 102,145 & 0,137 & 0,132 & 0,141 & 0,108 & 0,087 & 0,432 & 0,146 & 0,088 & 102,718 & 0,391&0,308 \\ 
		& 20\_20\_20\_1 & 311,688 & 4,922 & 4,596 & 0,882 & 0,358 & 0,117 & 9,657 & 0,600 & 0,291 & 322,228 & 5,879&5,004 \\ 
		& 3\_10\_3\_1 & 31,109 & 0,013 & 0,014 & 0,103 & 0,084 & 0,085 & 0,079 & 0,046 & 0,030 & 31,291 & 0,143&0,128 \\ 
		& 3\_5\_3\_1 & 0,298 & 0,006 & 0,008 & 0,086 & 0,081 & 0,080 & 0,042 & 0,043 & 0,026 & 0,427 & 0,129&0,113 \\ 
		\cline{2-14}
		& 20\_20\_20\_2 & 302,046 & 4,701 & 8,235 & 0,844 & 0,350 & 0,114 & 9,204 & 0,550 & 0,244 & 312,094 & 5,600&8,593 \\ 
		& 10\_10\_10\_2 & 168,211 & 0,148 & 0,216 & 0,146 & 0,112 & 0,089 & 0,459 & 0,149 & 0,087 & 168,816 & 0,409&0,392 \\ 
		& 3\_10\_3\_2 & 12,291 & 0,016 & 0,045 & 0,090 & 0,083 & 0,083 & 0,056 & 0,044 & 0,027 & 12,437 & 0,142&0,155 \\ 
		& 3\_5\_3\_2 & 0,130 & 0,007 & 0,015 & 0,088 & 0,080 & 0,082 & 0,035 & 0,044 & 0,026 & 0,252 & 0,131&0,122 \\ 
		\hline
		\hline
		\multirow{8}{*}{$\FF_3$}& 10\_10\_10\_1 & 300,266 & 0,247 & 0,116 & 0,102 & 0,216 & 0,097 & 0,221 & 1,294 & 0,056 & 300,589 & 1,758&0,269 \\ 
		& 20\_20\_20\_1 & 300,395 & 6,348 & 4,149 & 0,190 & 1,554 & 0,178 & 3,291 & 31,444 & 0,256 & 303,876 & 39,346&4,583 \\ 
		& 10\_3\_10\_1 & 300,270 & 0,062 & 0,027 & 0,099 & 0,193 & 0,092 & 0,134 & 0,940 & 0,050 & 300,503 & 1,194&0,169 \\ 
		& 5\_3\_5\_1 & 300,432 & 0,014 & 0,008 & 0,100 & 0,121 & 0,092 & 0,032 & 0,188 & 0,027 & 300,564 & 0,324&0,127 \\ 
		\cline{2-14}
		& 10\_10\_10\_2 & 300,194 & 0,279 & 0,125 & 0,098 & 0,224 & 0,104 & 0,221 & 1,416 & 0,056 & 300,513 & 1,919&0,284 \\ 
		& 20\_20\_20\_2 & 300,096 & 4,593 & 2,919 & 0,169 & 1,373 & 0,161 & 2,822 & 22,726 & 0,173 & 303,087 & 28,692 & 3,253 \\ 
		& 10\_3\_10\_2 & 300,422 & 0,074 & 0,032 & 0,103 & 0,207 & 0,100 & 0,118 & 1,015 & 0,057 & 300,643 & 1,295&0,189 \\ 
		& 5\_3\_5\_2 & 300,451 & 0,024 & 0,014 & 0,134 & 0,154 & 0,124 & 0,038 & 0,255 & 0,036 & 300,624 & 0,433&0,173 \\ 
	\end{tabular}
	\caption{Running times of the models }
	\label{tbl:TimesTable}
\end{table}
\end{landscape}

	\blue{
	\paragraph{Acknowledgments:} This publication was supported by the Institute of Operations Research at the Karlsruhe Institute of Technology (KIT). We also want acknowledge the members of the Vienna Graduate School of Computational Optimization (VGSCO) for their valuable discussion and input, in particular Prof. Günther Raidl, who suggested the experiments conducted in \cref{sec:Restricting Gurobi's time limit to the conic solution time}. Finally, we want to thank the anonymous referees for there diligence and for suggesting the experiments conducted in \cref{sec:Comparing the sparse and the full conic models}.
	}
	
	\bibliography{Literature}
	\bibliographystyle{abbrv}	
	\appendix

\end{document}